\documentclass[12pt,reqno]{article} 
\usepackage{amsmath,amssymb,amsthm,amsfonts, indentfirst}
\usepackage{enumerate,color,bm,graphicx,here}
\usepackage[T1]{fontenc}
\usepackage[utf8]{inputenc}

\usepackage{hyperref}
\newcommand{\tops}[2]{\texorpdfstring{#1}{#2}}

\hypersetup{
  colorlinks=true,
  linkcolor=black,
  filecolor=black,
  citecolor=black,
  urlcolor=blue,
  pdftitle={Large densities in a competitive two-species chemotaxis system in the non-symmetric case},
  pdfauthor={Shohei Kohatsu, Johannes Lankeit},
  pdfkeywords={},
  bookmarksopen=false,
}

\newcommand{\Om}{\Omega}
\usepackage{newunicodechar}
\newunicodechar{∇}{\nabla}
\newunicodechar{ℝ}{\mathbb{R}}
\newunicodechar{Δ}{\Delta}
\newunicodechar{μ}{\mu}
\newunicodechar{α}{\alpha}
\newunicodechar{β}{\beta}
\newunicodechar{ℕ}{\mathbb{N}}
\newunicodechar{κ}{\kappa}
\newunicodechar{λ}{\lambda}
\newunicodechar{γ}{\gamma}
\newunicodechar{χ}{\chi}
\newunicodechar{σ}{\sigma}
\makeatletter
\def\@cite#1#2{[{{\bfseries #1}\if@tempswa , #2\fi}]}
\renewcommand{\section}{%
\@startsection{section}{1}{\z@}
{0.5truecm plus -1ex minus -.2ex}%
{1.0ex plus .2ex}{\bfseries\large}}
\def\@seccntformat#1{\csname the#1\endcsname.\ }
\makeatother

\numberwithin{equation}{section} 
\newtheorem{thm}{Theorem}[section]

\newtheorem{lem}[thm]{Lemma}

\theoremstyle{definition}
\newtheorem{df}{Definition}[section]
\newtheorem{remark}{Remark}[section]
\newtheorem*{prth1.1}{Proof of Theorem 1.1}
\newtheorem*{prth1.2}{Proof of Theorem 1.2}

\newcommand{\ep}{\varepsilon}
\newcommand{\pa}{\partial}
\newcommand{\R}{\mathbb{R}}
\newcommand{\N}{\mathbb{N}}

\newcommand{\Tmax}{T_{\rm max}}
\newcommand{\Tmaxd}{T_{\rm max}^{d_1, d_2}}

\newcommand{\wsc}{\stackrel{\ast}{\rightharpoonup}}
\newcommand{\Ombar}{\overline{\Omega}}
\newcommand{\dt}{\dfrac{{\rm d}}{{\rm d}t}}
\newcommand{\norm}[2][]{\|#2\|_{#1}}
\newcommand{\Lom}[1]{L^{#1}(\Omega)}
\newcommand{\io}{\int_\Omega}
\newcommand{\set}[1]{\{#1\}}
\newcommand{\kl}[1]{\left(#1\right)}
\newcommand{\f}[2]{\frac{#1}{#2}}

\newcommand{\til}[1]{\widetilde{#1}}
\newcommand{\util}{\til{u}}
\newcommand{\vtil}{\til{v}}

\title{Large densities\\ in a competitive two-species chemotaxis system\\ in the non-symmetric case}

\author{Shohei Kohatsu\footnote{sho.kht.jp@gmail.com}\\
 \footnotesize   Department of Mathematics, 
    Tokyo University of Science\\
 \footnotesize   1-3, Kagurazaka, Shinjuku-ku, 
    Tokyo 162-8601, Japan\\
\footnotesize    ORCID: 0009-0005-8991-8970\\[12pt]
Johannes Lankeit\footnote{lankeit@ifam.uni-hannover.de}\\
\footnotesize Leibniz Universität Hannover, 
Institut für Angewandte Mathematik\\
\footnotesize Welfengarten 1, 
30167 Hannover, Germany\\
\footnotesize ORCID: 0000-0002-2563-7759
}%
\begin{document}
\maketitle

\textbf{Abstract}
This paper deals with the two-species chemotaxis system with Lotka--Volterra competitive kinetics,
\begin{align*}
  \begin{cases}
    u_t = d_1 \Delta u - \chi_1 \nabla \cdot (u \nabla w)
             + \mu_1 u (1 - u - a_1 v),
    & x\in\Omega,\ t>0, 
  \\
    v_t = d_2 \Delta v - \chi_2 \nabla \cdot (v \nabla w)
             + \mu_2 v (1 - a_2 u - v),
    & x\in\Omega,\ t>0,
  \\
    0 = d_3 \Delta w + \alpha u + \beta v - \gamma w,
    & x\in\Omega,\ t>0,
  \end{cases}
\end{align*}
under homogeneous Neumann boundary conditions and suitable initial conditions,
where $\Omega \subset \mathbb{R}^n$ $(n \in \mathbb{N})$ is a bounded domain
with smooth boundary,
$d_1, d_2, d_3, \chi_1, \chi_2, \mu_1, \mu_2 > 0$,
$a_1, a_2 \ge 0$ and $\alpha, \beta, \gamma > 0$.
Under largeness conditions on $\chi_1$ and $\chi_2$, we show that
for suitably regular initial data, any thresholds of the population density
can be surpassed,
which extends the previous results to the non-symmetric case.\\
The paper contains a well-posedness result for the hyperbolic--elliptic limit system with $d_1=d_2=0$.

\textbf{Key words:} Chemotaxis, logistic source, transient growth, well-posedness, blow-up
\textbf{MSC:} 35K55 (primary); 35B44; 35D30; 92C17 (secondary)
\section{Introduction}
Chemotaxis systems are used to describe aggregation
(usually in biological settings, caused by motion attracted to a self-produced signal chemical),
which often manifests in the striking form of finite-time blow-up.
In certain model variants
which have been introduced to consider biologically more realistic situations 
such as the inclusion of population or cell growth,
such blow-up is either not known to occur or known not to occur.
Nevertheless, even in such settings, large population densities can emerge.
The particular system that we address in this paper
is the following two-species 
chemotaxis system
with Lotka--Volterra competitive kinetics:
\begin{align}\label{2sp-LV}
  \begin{cases}
    u_t = d_1 \Delta u - \chi_1 ∇ \cdot (u ∇ w)
             + \mu_1 u (1 - u - a_1 v),
    & x\in\Omega,\ t>0, 
  \\
    v_t = d_2 \Delta v - \chi_2 ∇ \cdot (v ∇ w)
             + \mu_2 v (1 - a_2 u - v),
    & x\in\Omega,\ t>0,
  \\
    0 = d_3 \Delta w + \alpha u + \beta v - \gamma w,
    & x\in\Omega,\ t>0,
  \\
    ∇ u \cdot \nu = ∇ v \cdot \nu = ∇ w \cdot \nu = 0,
    & x\in \pa\Omega,\ t>0,
  \\
    u(x,0)=u_0(x),\ v(x,0)=v_0(x),
    & x\in\Omega,
  \end{cases}
\end{align}
where $\Omega \subset \R^n$ $(n \in \N)$ is a bounded domain
with smooth boundary $\pa\Omega$;
$d_1, d_2 > 0$, $d_3, \chi_1, \chi_2, \mu_1, \mu_2 > 0$, $a_1, a_2 \ge 0$
and $\alpha, \beta, \gamma > 0$ are constants; 
$\nu$ is the outward normal vector to $\pa\Omega$.
The initial data $u_0$ and $v_0$
are assumed to be nonnegative functions.

This model describes
a situation in which two species (with densities $u$ and $v$) compete for the resources
and move toward a higher concentration ($w$) of the signal substance
produced by themselves.
The large time behaviour of solutions to system \eqref{2sp-LV}
was initially studied by Tello and Winkler \cite{TW-2012}.

As in the classical Keller--Segel system (see \cite{KS-1970} or surveys \cite{BBTW,horstmann_survey,lan_win_survey}), where the question of whether blow-up occurs has been studied extensively (e.g. \cite{NS-1998, B-1998, N-2001} for the parabolic--elliptic and \cite{SS-2001, HW-2001, W-2013} for the fully parabolic system variant), also for the one-species analogue
\begin{align}\label{1sp}
  \begin{cases}
    u_t = \ep \Delta u - ∇ \cdot (u ∇ w)
     + \lambda u - \mu u^2,
    & x\in\Omega,\ t>0, 
  \\
    0 = \Delta w - w + u,
    & x\in\Omega,\ t>0
  \end{cases}
\end{align}
of \eqref{2sp-LV}, the possibility of blow-up has received some attention. For $n=2$, its fully parabolic version was first studied in \cite{MT-1996} and was proven to have global solutions \cite{OTYM-2002}, while \eqref{1sp} in this ``parabolic--elliptic'' form was treated in \cite{TW-2007}: If $n\le 2$ or, in higher dimension $n\ge 3$, $μ>\frac{n-2}{n}$, classical solutions exist, and according to \cite{KS-2016} this extends to $μ=\frac{n-2}{n}$.

On the other hand,
when the logistic source of \eqref{1sp} is replaced by $\lambda u - \mu u^{\kappa}$,
some solutions blow up in finite time
when $\kappa \in (1, \frac{7}{6})\ (n = 3,4)$
or $\kappa \in (1, 1 + \frac{1}{2(n-1)})\ (n \ge 5)$ \cite{W-2018}. 
Similar finite-time blow-up result for \eqref{1sp} with a modified second equation
was proved in \cite{F-2021}.
However, it is still not fully investigated whether solutions to \eqref{1sp}
may blow up or not.

Nevertheless: Blow-up is not the only possible manifestation of emerging structure, and regarding \eqref{2sp-LV} there are numerical observations suggesting the existence of large population densities at intermediate times (cf. \cite{BLM-2016}); in \eqref{1sp}, such possibly transient large population densities have been captured in proofs:
When $n = 1$, Winkler \cite{W-2014} showed that
if $\mu \in (0,1)$,
under an appropriate condition on the initial data,
one can find some time $T > 0$ with the property that
for each $M > 0$ there is $\ep_0 > 0$ such that
whenever $\ep \in (0, \ep_0)$, there is
$(x_{\ep}, t_{\ep}) \in \Omega \times (0,T)$ such that
the solution $(u_{\ep}, v_{\ep})$ to \eqref{1sp} satisfies
\begin{align}\label{TGP}
    u_{\ep}(x_{\ep}, t_{\ep}) > M.
\end{align}
Later on, this result was extended to
higher-dimensional case under the condition that
$\Omega$ is a ball 
and the solutions are radially symmetric~\cite{L-2015}. These results mean that
if the death rate $\mu$ is small enough
and the diffusion is weak,
any thresholds of the population density can be surpassed.
We note that this can happen
even in the cases that solutions are always global and bounded.

The key to the proofs in \cite{W-2014, L-2015} lies in considering 
the following system, 
which can be obtained by letting $\ep \to 0$ in \eqref{1sp}:
\begin{align}\label{1sp-HE}
  \begin{cases}
    u_t = - ∇ \cdot (u ∇ w)
     + \lambda u - \mu u^2,
    & x\in\Omega,\ t>0, 
  \\
    0 = \Delta w - w + u,
    & x\in\Omega,\ t>0,
  \end{cases}
\end{align}
and a blow-up result in \eqref{1sp-HE} was obtained,
which leads to the transient growth phenomena \eqref{TGP}
for the system \eqref{1sp}.
This system was also studied by Kang and Stevens \cite{KS-2016} who due to a crucial insight could avoid any requirement of radial symmetry and obtained a blow-up result in the cases that $\Om=ℝ^n$ or that $\Om\subset ℝ^n$ is a bounded convex domain.

Also global existence and boundedness of solutions of \eqref{2sp-LV}
have been widely studied.
For instance, when $\alpha = \beta = \gamma = 1$ and $a_1, a_2 > 0$,
Lin et al.~\cite{LMZ-2018} established global boundedness in the case that
$\frac{\chi_1}{\mu_1} + \frac{\chi_2}{\mu_2} < d_3$;
when $\alpha, \beta, \gamma > 0$, Mizukami~\cite{M-2018} obtained the result
under the conditions that
\begin{align}\label{M2018}
  n=2, \quad\mbox{or}\quad
  n \ge 3, \ \frac{\chi_1}{\mu_1} 
    < \frac{n d_3}{n-2} \min\left\{\frac{1}{\alpha}, \frac{a_1}{\beta}\right\}
  \ \mbox{and}\ 
  \frac{\chi_2}{\mu_2} < \frac{n d_3}{n-2}
    \min\left\{\frac{1}{\beta}, \frac{a_2}{\alpha}\right\}.
\end{align}
Similar results regarding global existence and boundedness were obtained
by Wang~\cite{W-2020},
where the conditions on parameters were improvement of those in \cite{LMZ-2018}.
They also extended parameter ranges for $\frac{\chi_1}{\mu_1}$ and
$\frac{\chi_2}{\mu_2}$ in \cite{M-2018}
when $a_1 \le \frac{n}{n-2}\cdot \frac{\beta}{\alpha}$ as well as
$a_2 \le \frac{n}{n-2}\cdot \frac{\alpha}{\beta}$;
more precisely, they showed global boundedness when
\[
    \frac{a_1 d_3}{\beta} < \frac{\chi_1}{\mu_1} < \frac{d_3}{\alpha}, \ 
    \frac{a_2 d_3}{\alpha} < \frac{\chi_2}{\mu_2} < \frac{d_3}{\beta} \ \mbox{and}\ 
    \frac{\alpha - \beta a_2}{d_3} \cdot \frac{\chi_1}{\mu_1}
    + \frac{\beta - \alpha a_1}{d_3} \cdot \frac{\chi_2}{\mu_2}
    + a_1 a_2 < 1.
\]
%
On the other hand, when the terms $\mu_1 u (1 - u - a_1 v)$ and
$\mu_2 v (1 - a_2 u - v)$ in \eqref{2sp-LV} are replaced by related terms with smaller exponents,
$\mu_1 u (1 - u^{\kappa_1 - 1} - a_1 v^{\lambda_1 - 1})$ and
$\mu_2 v (1 - a_2 u^{\lambda_2 - 1} - v^{\kappa_2 - 1})$, respectively,
Mizukami et al.~\cite{MTY-2022} established the possibility of finite-time blow-up
when $1 < \kappa_1, \lambda_1, \kappa_2, \lambda_2
< \min\{\frac{7}{6}, 1 + \frac{1}{2(n-1)}\} \ (n \ge 3)$.

Returning to the setting of $κ_1=κ_2=λ_1=λ_2=2$, which is not covered by the available blow-up results, but corresponds to the classical Lotka-Volterra competition terms, we are interested in
whether transient growth like \eqref{TGP}
occurs in the system \eqref{2sp-LV}.
However, 
unlike in the case of the one-species problem,
the $L^p$-estimates of $u(\cdot,t)$ and $v(\cdot,t)$ in \eqref{2sp-LV}
affect each other not only via the joint signal substance, but more directly because of the absorptive competition terms $- a_1 uv$, $- a_2 u v$
and this makes it difficult to analyze each species individually. Note that absorption should have a tendency to keep solutions small, and hence counteract any blow-up or emergent large densities. Still, the latter may develop: Recently, this question has been investigated by Xu \cite{X-2020} by tracking the sum of densities $u + v$.
Under the assumption of radial symmetry they studied the following 
hyperbolic--hyperbolic--elliptic system
\begin{align}\label{HHE-LV}
  \begin{cases}
    u_t = - \chi_1 ∇ \cdot (u ∇ w)
             + \mu_1 u (1 - u - a_1 v),
    & x\in\Omega,\ t>0, 
  \\
    v_t = - \chi_2 ∇ \cdot (v ∇ w)
             + \mu_2 v (1 - a_2 u - v),
    & x\in\Omega,\ t>0,
  \\
    0 = d_3 \Delta w + \alpha u + \beta v - \gamma w,
    & x\in\Omega,\ t>0,
  \\
    ∇ u \cdot \nu = ∇ v \cdot \nu = ∇ w \cdot \nu = 0,
    & x\in \pa\Omega,\ t>0,
  \\
    u(x,0)=u_0(x),\ v(x,0)=v_0(x),
    & x\in\Omega
  \end{cases}
\end{align}
which can be obtained by letting $d_1, d_2 \to 0$ in \eqref{2sp-LV},
and established conditions which implies that
larger initial data and small measure of domain ensure
blow-up results in \eqref{HHE-LV}.
After that,
they used this result to earn a property like \eqref{TGP} 
for the system \eqref{2sp-LV}.

Apart from notable improvements over \cite{X-2020} regarding the admissible parameter range (see Remark~\ref{condition:remark}), we avoid the symmetry condition on initial data -- which was needed in \cite{L-2015,X-2020} -- and convexity assumptions on the domain (\cite{KS-2016},
also used in \cite{X-2023-1}).

Our first main result reads as follows.
%
%
\begin{thm}\label{main_thm}
Let $n\in ℕ$, and let $\Omega \subset \R^n$ be a bounded domain with smooth boundary.
Let $d_3, \chi_1, \chi_2, \mu_1, \mu_2 > 0$,
$a_1, a_2 \ge 0$ and
$\alpha, \beta, \gamma > 0$.
Then for all $p > 1$ which satisfy one of the following
%
%
%
%
\begin{align}
&{\rm (i)} &
&\chi_1 > \dfrac{p d_3}{p-1} \mu_1 \max\left\{\dfrac{1}{\alpha},
\dfrac{a_1}{\beta}\right\}
\quad\mbox{and}\quad
\chi_2 > \dfrac{p d_3}{p-1}\mu_2 \max\left\{\dfrac{1}{\beta},
\dfrac{a_2}{\alpha}\right\}, \label{condition:p:1}
\\[1mm]
&{\rm (ii)} &
&
\begin{cases}
\chi_1 > \max\left\{\dfrac{d_3 p (p+1) \mu_1 + d_3 p a_2 \mu_2 - \chi_2 \alpha (p-1)}
{\alpha(p^2 - 1)}, \dfrac{d_3 p a_1 \mu_1}{\beta (p-1)}\right\}\quad\mbox{and}
\\[3mm]
\dfrac{d_3 p (p+1) \mu_2 + d_3 p^2 a_2 \mu_2}{\alpha p(p-1) + \beta (p^2-1)}
< \chi_2 \le \dfrac{d_3 p a_2 \mu_2}{\alpha (p-1)},
\end{cases} \label{condition:p:2}
\\[1mm]
&{\rm (iii)} &
&
\begin{cases}
\dfrac{d_3 p (p+1) \mu_1 + d_3 p^2 a_1 \mu_1}{\alpha (p^2-1) + \beta p(p-1)}
< \chi_1 \le \dfrac{d_3 p a_1 \mu_1}{\beta (p-1)}\quad\mbox{and}
\\[3mm]
\chi_2 > \max\left\{\dfrac{d_3 p a_1 \mu_1 + d_3 p (p + 1) \mu_2 - \chi_1 \beta (p-1)}
{\beta(p^2 - 1)}, \dfrac{d_3 p a_2 \mu_2}{\alpha (p-1)}\right\},
\end{cases} \label{condition:p:3}
\\[1mm]
&{\rm (iv)} &
&
\begin{cases}
\dfrac{d_3 p (p+1)\mu_1 + d_3 p^2 a_1 \mu_1 + d_3 p a_2 \mu_2 - \chi_2 \alpha(p-1)}
{\alpha (p^2-1) + \beta p (p-1)}
< \chi_1 \le \dfrac{d_3 p a_1 \mu_1}{\beta (p-1)}
\quad\mbox{and}
\\[3mm]
\dfrac{d_3 p (p+1)\mu_2 + d_3 p a_1 \mu_1 + d_3 p^2 a_2 \mu_2 - \chi_1 \beta(p-1)}
{\alpha p (p-1) + \beta (p^2-1)}
< \chi_2 \le \dfrac{d_3 p a_2 \mu_2}{\alpha (p-1)},
\end{cases} \label{condition:p:4}
\end{align}
%
there is $C(p) > 0$ with the following property\/{\rm :}
Whenever $q > \max\{n, 2\}$ and
$u_0, v_0 \in W^{1,q}(\Omega)$ are nonnegative such that
\begin{align}\label{condition:initial}
    \left(\io u_0^p + \io v_0^p \right)^{\frac{1}{p}}
    > C(p) \max\left\{1,\io u_0 + \io v_0\right\},
\end{align}
there is $\Tmax \in (0, \infty]$ and for each $d_1,d_2>0$ there is a uniquely determined triple $(u,v,w)$ of functions
which solves \eqref{2sp-LV} classically in $\Omega \times (0, \Tmax)$,
and moreover,
there exists $T \in (0, \Tmax)$ such that
to each $M > 0$ there corresponds some $d_{*}(M) > 0$
with the property that
for any $d_1, d_2 \in (0, d_{*}(M))$ one can find
$(x,t) \in \Omega \times (0, T)$ such that
$(u,v,w)$ satisfies
\[
    u(x,t) + v(x,t) > M.
\]
\end{thm}

%
%
\begin{remark}
If $a_1=a_2=d_3 = \chi_1 = \chi_2 = \alpha = \beta = \gamma = 1$,
then the conditions \eqref{condition:p:1}--\eqref{condition:p:4} 
can be reduced to
%
\begin{equation}\label{reduced}
 μ_1< \frac{p-1}{p} \quad\mbox{and}\quad
 μ_2< \frac{p-1}{p}
\end{equation}
which can be fulfilled for some $p>1$ if and only if
$\max\{μ_1,μ_2\}<1$.
We note that for $n \ge 3$, the conditions $\min\{μ_1,μ_2\}>\f{n-2}n$ for boundedness of solutions to \eqref{2sp-LV} are compatible with this requirement. Large densities as in Theorem~\ref{main_thm} hence have a chance to emerge
even when the solutions to \eqref{2sp-LV} are always global and bounded.
\end{remark}

%
%
\begin{remark}\label{withoutp}
We note that there is $p \in(1,\infty) $ such that
\eqref{condition:p:1}, \eqref{condition:p:2}, \eqref{condition:p:3} or
\eqref{condition:p:4} holds if and only if one of the following conditions is fulfilled:
\begin{align}
&{\rm (i)} &
&\chi_1 > d_3 \mu_1 \max\left\{\dfrac{1}{\alpha},
\dfrac{a_1}{\beta}\right\}
\quad\mbox{and}\quad
\chi_2 > d_3\mu_2 \max\left\{\dfrac{1}{\beta},
\dfrac{a_2}{\alpha}\right\}, \label{condition:1}
\\[1mm]
&{\rm (ii)} &
&
\chi_1 > d_3 \mu_1 \max\left\{\dfrac{1}{\alpha},
\dfrac{a_1}{\beta}\right\},\quad
\chi_2 > \frac{d_3 \mu_2 (1 + a_2)}{\alpha + \beta}
\quad\mbox{and}\quad
\frac{a_2}{\alpha} > \frac{1}{\beta}, \label{condition:2}
\\[1mm]
&{\rm (iii)} &
&
\chi_1 > \frac{d_3 \mu_1 (1 + a_1)}{\alpha + \beta},\quad
\chi_2 > d_3\mu_2 \max\left\{\dfrac{1}{\beta},
\dfrac{a_2}{\alpha}\right\}
\quad\mbox{and}\quad
\frac{a_1}{\beta} > \frac{1}{\alpha}, \label{condition:3}
\\[1mm]
&{\rm (iv)} &
&
\chi_1 > \frac{d_3 \mu_1 (1 + a_1)}{\alpha + \beta},\quad
\chi_2 > \frac{d_3 \mu_2 (1 + a_2)}{\alpha + \beta}, \quad
\frac{a_1}{\beta} > \frac{1}{\alpha}
\quad\mbox{and}\quad
\frac{a_2}{\alpha} > \frac{1}{\beta}.\label{condition:4}
\end{align}
In particular, if
$a_1=a_2=d_3=\mu_1=\mu_2=\alpha=β=γ=1$, then
these conditions become $\chi_1 > 1$ and $\chi_2 > 1$.%
\end{remark}
%
%
%
\begin{remark}\label{condition:remark}
The conditions \eqref{condition:p:1}--\eqref{condition:p:4} 
play an important rôle
on determining whether blow-up in \eqref{HHE-LV} occurs. 
In \cite[Theorem 1.2]{X-2020}, the corresponding conditions were
\begin{align}
  &\chi_1 >
  \frac{d_3 \mu_1 p (p + 1) + d_3 a_1 \mu_1 p^2 + d_3 a_2 \mu_2 p}
         {\alpha (p^2 - 1) - \gamma p (p-1)}
         \label{Xu1}
\\
    \intertext{and}
  &\chi_2 >
  \frac{d_3 \mu_2 p (p + 1) + d_3 a_2 \mu_2 p^2 + d_3 a_1 \mu_1 p}
         {\beta (p^2 - 1) - \gamma p (p-1)}.
         \label{Xu2}
\end{align}
Apparently, in Theorem~\ref{main_thm} smaller values of $χ_1$ and $χ_2$ (or larger values of $μ_1,μ_2$) can be chosen; also in the conditions on $p$, an upper bound is unnecessary in our theorem. For ease of comparison again turning to the setting of $a_1=a_2=d_3=χ_1=χ_2=α=β=γ=1$, we see that \eqref{Xu1} and \eqref{Xu2} become 
\[
 1-\f1p > μ_1+μ_2 +2p\max\{μ_1,μ_2\}
\]
(which at the very least necessitates $μ_1+μ_2<\max_{p>0} \f{p-1}{p^2+p} = \f{\sqrt{2}}{4+3\sqrt{2}}\approx 0{.}17$.)
These improvements will be achieved by changing the way of estimate
in Lemma~\ref{blowup:pre},
where we combine some inequalities to derive
\[
  \dfrac{\chi_1 \gamma (p-1)}{d_3 p} \int_0^t \io u^p w
  \le \eta \int_0^t \io u^{p+1} + \eta \int_0^t \io v^{p+1}
  + C(\eta)
\]
with sufficiently small $\eta > 0$,
so that we can remove $\gamma$
from \eqref{Xu1} and \eqref{Xu2}.
\end{remark}

\begin{remark}\label{remark:Xu-recent}
Very recently, in \cite{X-2023-2}, a related two-species system with signal loop was studied, where two species each produce a signal to which both may react. We denote all quantities from \cite{X-2023-2} with additional tildes. If $\tilde{σ}_1=\tilde{σ}_2=1$, \eqref{2sp-LV} is a special case of \cite[(3)]{X-2023-2} with the choices of $u=\tilde{u}$, $v=\tilde{w}$, $w=\tilde{v}+\f{\tilde{χ}_2}{\tilde{χ}_1}\tilde{z}$,
$\tilde{\chi}_1 = \chi_1$,
$\tilde{χ}_4=\chi_2$, $\tilde{χ}_3=χ_2\frac{\tilde{χ}_2}{\tilde{χ}_1}$,
$\tilde{\mu}_1 = \mu_1$, $\tilde{\mu}_2 = \mu_2$, $\tilde{a}_1 = a_1$,
$\tilde{a}_2 = a_2$, $d_3 = \beta = \gamma$ as well as
$\alpha = \gamma \frac{\tilde{\chi}_2}{\tilde{\chi}_1}$,
and the conditions on $\tilde{\chi}_2$ and $\tilde{\chi_4}$ that
ensure emergent large densities were
\begin{align}\label{con:loop}
    \tilde{\chi}_2 > \frac{p}{p-1}\left(\tilde{\mu}_1 + 
    \frac{\tilde{a}_1 \tilde{\mu}_1 p + \tilde{a}_2 \tilde{\mu}_2}
    {p + 1}\right)
      \quad\mbox{and}\quad
    \tilde{\chi}_4 > \frac{p}{p-1}\left(\tilde{\mu}_2 + 
    \frac{\tilde{a}_2 \tilde{\mu}_2 p + \tilde{a}_1 \tilde{\mu}_1}
    {p+1}\right)
\end{align}
with some $p > 1$
(see \cite[(15)]{X-2023-2}).
Again choosing the parameters as
$a_1=a_2=d_3=\mu_1=\mu_2=β=γ=1$, we see that then 
\eqref{con:loop} holds if and only if
\begin{equation}\label{conditionsXuSpecialcase}
    \tilde{\chi}_2 > 2
    \quad\mbox{and}\quad
    \tilde{\chi}_4 > 2,
\end{equation}
while \eqref{condition:1}--\eqref{condition:4} becomes
\begin{alignat}{2}
&{\rm (i)} \quad&
&\tilde{\chi}_2 > 1, \quad
\tilde{\chi}_4 > 1
\quad\mbox{and}\quad
\tilde{\chi}_2 \tilde{\chi}_4 > \tilde{\chi}_1 > 1,\label{ourcondition1}
\\
&{\rm (ii)} \quad&
&\tilde{\chi}_1 > \tilde{\chi}_2 > 1
\quad\mbox{and}\quad
\tilde{\chi}_4 > \frac{2\tilde{\chi}_1}{\tilde{\chi_1} + \tilde{\chi}_2},\label{ourcondition2}
\\
&{\rm (iii)} \quad&
&\tilde{\chi}_2 > 1,\quad
\tilde{\chi}_4 > 1,\quad
\tilde{\chi}_2 > \tilde{\chi}_1
\quad\mbox{and}\quad
\tilde{\chi}_2 + \tilde{\chi}_1 > 2,\label{ourcondition3}
\end{alignat}
(whereas (iv) only ever can be satisfied for different parameter choices).
%
%
These show that we (in \eqref{ourcondition1}--\eqref{ourcondition3}) can actually take $\tilde{\chi}_2$ and $\tilde{\chi}_4$ smaller than
in \cite{X-2023-2}, i.e. \eqref{conditionsXuSpecialcase}.
\end{remark}

%
%
\begin{remark}
Following the arguments in \cite[Sections 3.2 and 4.6]{L-2015},
we can observe that the system \eqref{HHE-LV} admits a global solution
under the condition that
\[
    \frac{\chi_1}{\mu_1}
    \le d_3 \min\left\{\frac{1}{\alpha}, \frac{a_1}{\beta}\right\}
    \quad\mbox{and}\quad
    \frac{\chi_2}{\mu_2}
    \le d_3 \min\left\{\frac{1}{\beta}, \frac{a_2}{\alpha}\right\}.
\]
If $a_1=a_2=d_3 = \mu_1 = \mu_2 = \alpha = \beta = \gamma = 1$, this condition
can be reduced to $\chi_1, \chi_2 \le 1$.
In contrast, for the same parameters, the condition on $\chi_1$ and $\chi_2$
for blow-up in \eqref{HHE-LV} is $\chi_1 > 1$ and $\chi_2 > 1$
(see Remark~\ref{withoutp}). This suggests that the given condition for the possibility of blow-up in \eqref{HHE-LV}
and thus for large densities in \eqref{2sp-LV} is not too far from being sharp.
\end{remark}

%
%
\begin{remark}
The condition for initial data \eqref{condition:initial}
seems to be a generalization of the one in \cite[Theorem 1.1]{L-2015}.
Indeed, if we choose $v_0 \equiv 0$, then \eqref{condition:initial}
can be reduced to
\[
  \left(\io u_0^p\right)^{\frac{1}{p}}
    > C(p) \max\left\{1, 
        \io u_0\right\},
\]
which is equivalent to the condition in \cite{L-2015} under $\kappa = \mu$.
\end{remark}

%
%
\begin{remark}
In \cite{L-2015,X-2020}, initial data had to satisfy higher regularity assumptions, including a compatibility condition. This is no longer needed, see Section~\ref{subsec:locex} for a short discussion. 
\end{remark}

The proof of the main theorem is based on \cite{L-2015}. 
Following said work, we will consider the system \eqref{HHE-LV}
in $\Omega \times (0, T)$ for some $T > 0$,
and show existence of solutions to \eqref{HHE-LV}
which can be approximated by solutions of \eqref{2sp-LV}
(Lemma~\ref{exist:app}).
Then we will prove that
under some conditions for parameters and initial data,
any solution to \eqref{HHE-LV} blows up in finite time
(Lemma~\ref{blowup:result}).
From this, we can derive the main theorem
which tells that any thresholds of
the sum of population densities can be surpassed,
since the solutions of \eqref{HHE-LV} can be obtained as
limits of solutions to \eqref{2sp-LV}.

In order to establish blow-up result in the system \eqref{HHE-LV},
it will be important to obtain a blow-up criterion to \eqref{HHE-LV}.
Our second result is concerned with local well-posedness of \eqref{HHE-LV}.
%
%
\begin{thm}\label{thm2}
Let $n\in ℕ$, and let $\Omega \subset \R^n$ be a bounded domain with smooth boundary.
Let $d_3, \chi_1, \chi_2, \mu_1, \mu_2 > 0$,
$a_1, a_2 \ge 0$,
$\alpha, \beta, \gamma > 0$,
and $q > \max\{n, 2\}$.
Then for all nonnegative functions $u_0, v_0 \in W^{1,q}(\Omega)$,
there are $\Tmax \in (0, \infty]$ and a uniquely determined triple $(u, v, w)$
of functions
\[
\begin{cases}
  u, v \in C^0(\Ombar \times [0, \Tmax)) \cap
     L^{\infty}_{{\rm loc}}([0, \Tmax); W^{1,q}(\Omega))
       \quad\mbox{and}
\\
   w \in C^{2,0}(\Ombar \times [0, \Tmax))
\end{cases}
\]   
such that $(u, v, w)$ is a $W^{1,q}$-solution of \eqref{HHE-LV} in
$\Omega \times (0, \Tmax)$ in the sense of Definition~\ref{df:W1q},
and which are such that
\begin{align}\label{thm:HHEcriterion}
  \mbox{either}\ \Tmax = \infty
  \quad\mbox{or}\quad
  \limsup_{t \nearrow \Tmax}
    (\|u(\cdot, t)\|_{L^{\infty}(\Omega)}
      + \|v(\cdot, t)\|_{L^{\infty}(\Omega)}) = \infty.
\end{align}
In addition, let $u_0, v_0 \in W^{1,q}(\Omega)$
be nonnegative functions, and let $(u_{0j})_{j \in \N} \subset W^{1,q}(\Omega)$,
$(v_{0j})_{j \in \N} \subset W^{1,q}(\Omega)$
be sequences of nonnegative functions such that
\begin{align}\label{ini:con}
    \|u_{0j} - u_0\|_{L^q(\Omega)} \to 0
    \quad\mbox{and}\quad
    \|v_{0j} - v_0\|_{L^q(\Omega)} \to 0
\end{align}
as $j \to \infty$.
Let $(u, v, w)$, $(u_j, v_j, w_j)$ be the $W^{1,q}$-solution of \eqref{HHE-LV}
in $\Omega \times (0, T)$ for some $T > 0$ with initial data $u_0, v_0$,
and $u_{0j}, v_{0j}$, respectively.
Then for all $t \in (0, T)$,
\begin{align}\label{sol:con}
    \|u_j(\cdot, t) - u(\cdot, t)\|_{L^q(\Omega)}
    + \|v_j(\cdot, t) - v(\cdot, t)\|_{L^q(\Omega)}
    + \|w_j(\cdot, t) - w(\cdot, t)\|_{W^{2,q}(\Omega)}
    \to 0
\end{align}
as $j \to \infty$.
\end{thm}
The criterion \eqref{thm:HHEcriterion} will be discussed in Lemma~\ref{HHE-local}.
To achieve this,
we aim to derive important estimates of solutions to \eqref{2sp-LV}
in Section~\ref{W1,q-Linf},
which shows that $W^{1,q}$-norm of solutions can be controlled by
their $L^{\infty}$-norm.
Here, the difficulties of the non-symmetric case appeared
(Lemma~\ref{control_grad});
we have to treat integrals like below:
\begin{align}\label{diff_int}
    \io |∇u|^{q-2} ∇u \cdot ∇ \Delta u
  \quad\mbox{and}\quad
    \io |∇u|^{q-2} ∇u \cdot ∇(∇u \cdot ∇w),
\end{align}
which arise from the diffusion term $\Delta u$
and the chemotactic term $-∇ \cdot(u∇w)$, respectively.
In \cite{L-2015, X-2020},
the estimates of corresponding terms heavily rely on radial symmetric
of solutions,
and thus we have to modify their proofs.
The keys of the estimates are
a result on the boundary integration which was
established in
\cite{ISY-2014, LW-2017},
and some useful inequalities related to BMO theory
from \cite{KT-2000, OT-2003, S-1993, WYW-2006}. 
For the non-symmetric case of the two-species two-stimuli chemotaxis system in \cite{X-2023-2}, the method of the proof is similar; 
for a comparison of results see Remark~\ref{remark:Xu-recent}.
%
%
\begin{remark}
In \cite{KS-2016}, convexity of the domain was required in order to
estimate the boundary integral, which comes from the first integral in \eqref{diff_int},
as
\begin{align*}
  \int_{\pa\Omega} |∇u|^{q-2}(\nabla |\nabla u|^2 \cdot \nu) \le 0.
\end{align*}
In contrast, we will follow arguments as in \cite{ISY-2014}
to get rid of convexity.
As was mentioned in \cite[Remark 2]{KS-2016},
we do not have a uniform bound on $\| |∇u|^q \|_{L^{\frac{s}{q}}(\Omega)}$
for $\frac{s}{q} > 1$, however,
we use $\| |∇u|^q \|_{L^{\frac{2}{q}}(\Omega)}$ instead so that
we can derive an differential inequality for
$\| ∇u \|_{L^{q}(\Omega)}^q$.
This is indeed possible by choosing $a \in (0,1)$ large enough in
\cite[Lemma 2.5]{ISY-2014}.
The result of this observation was also obtained in \cite[Lemma 2.1.\ c)]{LW-2017}.
\end{remark}

This paper is organized as follows.
In Section~\ref{Pre}, we give some properties of elliptic problems,
introduce some facts related to BMO,
and prepare an estimate (Lemma~\ref{D2w})
which helps to derive $W^{1,q}$-estimates for solutions
of \eqref{2sp-LV} afterward
without symmetry assumptions.
In Section~\ref{PPE}, we will focus on existence of solutions
to \eqref{2sp-LV} and estimates
until a common existence time (Lemma~\ref{2sp-enough}),
and prepare some estimates for compactness arguments
(Lemma~\ref{utvt}).
Section~\ref{HHE} is devoted to definition, local existence,
estimates
for solutions to \eqref{HHE-LV},
and proving Theorem~\ref{thm2}.
Finally, in Section~\ref{conclude},
we consider blow-up of solutions to \eqref{HHE-LV},
and prove Theorem~\ref{main_thm}.

\section{Preliminaries}\label{Pre}

We recall some estimates for solutions of elliptic problems.
In Lemma~\ref{D2w}, they will culminate in the $L^\infty$-estimate for the second derivative of solutions $w$ of \eqref{2sp-LV} and \eqref{HHE-LV}. We begin, however, with a different simple $L^p$-estimate:
%
%
\begin{lem}\label{Lp-elli}
Let $\Omega \subset \R^n$ $(n \in \N)$ be a bounded domain with smooth boundary.
Let $u, v \in L^{\infty}(\Omega)$,
$\alpha, \beta, \gamma > 0$
and $w \in C^2(\Ombar)$ be a nonnegative classical solution of
\begin{align}\label{pr:elli}
  \begin{cases}
    - \Delta w + \gamma w = \alpha u + \beta v
    & \mbox{in}\ \Omega,
  \\
    ∇ w \cdot \nu = 0
    & \mbox{in}\ \pa\Omega.
  \end{cases}
\end{align}
Then for all $p \in [1, \infty]$,
\[
  \|w \|_{L^p(\Omega)}
  \le \dfrac{\alpha}{\gamma} \|u \|_{L^p(\Omega)}
   + \dfrac{\beta}{\gamma} \|v\|_{L^p(\Omega)}.
\]
\end{lem}
\begin{proof}
We test the equation by $w^{p-1}$
to obtain the estimate for $p \in (1, \infty)$.
Then we can extend the result to $p\in\set{1, \infty}$
by taking the limits $p \to 1$ and $p \to \infty$, respectively.
\end{proof}

We next recall maximal regularity for elliptic problems,
which makes it possible to derive an estimate
for the derivatives of $w$ in terms of $u$ and $v$.

%
%
\begin{lem}\label{max_reg:elli}
Let $\Omega \subset \R^n$ $(n \in \N)$ be a bounded domain with smooth boundary.
Let $\alpha, \beta, \gamma > 0$.
For any $q > 1$ and $\sigma > 0$,
there exists a constant $C > 0$ such that
the following hold\/{\rm :}
\begin{enumerate}[{\rm (i)}]
\item For every $u, v \in L^{\infty}(\Omega)$
every classical solution $w \in C^2(\Ombar)$
of \eqref{pr:elli} satisfies
\[
    \|w\|_{W^{2,q}(\Omega)}
    \le C(\|u\|_{L^q(\Omega)} + \|v\|_{L^q(\Omega)}).
\]
In addition, if $u, v \in W^{1,q}(\Omega)$, then
\[
  \|w\|_{W^{3,q}(\Omega)}
  \le C(\|u\|_{W^{1,q}(\Omega)} + \|v\|_{W^{1,q}(\Omega)}).
\]
%
\item For every $u, v \in C^{\sigma}(\Ombar)$
any classical solution
$w \in C^{2+\sigma}(\Ombar)$ of \eqref{pr:elli} satisfies
\[
    \|w\|_{C^{2+\sigma}(\Omega)}
    \le C(\|u\|_{C^{\sigma}(\Omega)} + \|v\|_{C^{\sigma}(\Omega)}).
\]
\end{enumerate}
\end{lem}
\begin{proof}
For the proof,
see \cite[Theorem 15.2]{ADN-1959} and
\cite[Theorem 6.30]{GT-2001}
(see also \cite{F-1976}).
\end{proof}

The maximal regularity (Lemma~\ref{max_reg:elli}) is known to fail
for $q = \infty$,
so we cannot derive an $L^{\infty}$-estimate
of the second derivatives of $w$ from this lemma.
This may cause problems when we consider
the case of functions without radial symmetry,
especially in Section~\ref{W1,q-Linf},
where we have to deal with the term $D^2 w$.

To conquer this difficulty
and extend the result to the non-symmetric case,
the following lemma is very important. 
It has been used in \cite[Lemma 6]{KS-2016} for related models.

%
%
\begin{lem}\label{bmo:elli}
Let $\Omega \subset \R^n$ $(n \in \N)$ be a bounded domain with smooth boundary.
Let $\alpha > 0$ and $\beta \ge 0$.
Then there is $C > 0$ such that
for every $f \in L^{\infty}(\Omega)$
any weak solution $\phi \in W^{1,2}(\Omega)$ of
\[
    \begin{cases}
    - \Delta \phi + \alpha \phi = \beta f
    & \mbox{in}\ \Omega,
  \\
    ∇ \phi \cdot \nu = 0
    & \mbox{in}\ \pa\Omega
  \end{cases}
\]
satisfies
\[
  \| D^2 \phi \|_{{\rm BMO(\Omega)}} \le C\|f\|_{L^{\infty}(\Omega)}.
\]
\end{lem}
\begin{proof}
Regularity estimates in the Campanato space $\mathcal{L}^{2,n}(\Om)$ (like \cite[Propositions 9.1.1 and 9.1.2]{WYW-2006}) are combined with the embedding $\mathcal{L}^{2,n}(\Omega) \hookrightarrow {\rm BMO}(\Omega)$
(cf.~\cite[p.~259]{WYW-2006}); alternatively, mapping properties of the singular integral operator mapping $f$ to the second derivative of $\phi$
(\cite[p.~178, ${6.3}^a$]{S-1993})
can be exploited
(cf.~\cite[Lemma 6]{KS-2016}).
\end{proof}


In order to use Lemma~\ref{bmo:elli},
the following logarithmic Sobolev inequality
established in \cite{OT-2003} is also useful.

%
%
\begin{lem}\label{log:Sobolev}
(Cf. {\rm \cite[Lemma 2.8]{OT-2003}},
see also {\rm \cite[Theorem 1]{KT-2000}})
Let $\Omega \subset \R^n$ $(n \in \N)$ be a bounded domain with smooth boundary.
For any $q > n$, there exists a constant $C > 0$ such that
\[
    \|f\|_{L^{\infty}(\Omega)}
    \le C\big(1 + \|f\|_{{\rm BMO(\Omega)}}
      (1 + \log^{+} \|f\|_{W^{1,q}(\Omega)})\big)
\]
for all $f \in W^{1,q}(\Omega)$,
where $\log^{+} z = z$ for $z > 1$ and $\log^{+} z = 0$ for otherwise.
\end{lem}
\begin{proof}
When $n = 3$, this is already proved in \cite[Lemma 2.8]{OT-2003}.
Copying the proof in \cite{OT-2003} with replacing $\gamma$ by $1 - \frac{n}{q}$,
we can see this actually holds for any $n \in \N$.
\end{proof}

Now, we end this section by stating the follow result
as an application of
Lemmas~\ref{max_reg:elli}, \ref{bmo:elli} and \ref{log:Sobolev},
which allows us to estimate the term $D^2 w$ for solutions $w$ of \eqref{pr:elli}.

%
%
\begin{lem}\label{D2w}
Let $\Omega \subset \R^n$ $(n \in \N)$ be a bounded domain with smooth boundary.
For any $q > n$
and $K > 0$ there is $C  (= C(q, K)) > 0$ such that
for every $u, v \in W^{1,q}(\Omega)$
satisfying $\norm[\Lom1]{u}\le K$ and $\norm[\Lom1]{v}\le K$,
every classical solution $w \in C^2(\Ombar)$ of \eqref{pr:elli}
satisfies
\[
    \|D^2 w\|_{L^{\infty}(\Omega)}
    \le C\left(1 +
      (\|u\|_{L^{\infty}(\Omega)} + \|v\|_{L^{\infty}(\Omega)})
      (1 + \log
      (1 + \|∇u\|_{L^q(\Omega)}^q + \|∇v\|_{L^q(\Omega)}^q)
      ) \right).
\]
\end{lem}
\begin{proof}
We first note that
Lemma~\ref{max_reg:elli} provides $C_1 > 0$ such that
\begin{align}\label{Singular2}
    \|D^2 w\|_{W^{1,q}(\Omega)}
    \le C_1 (\|u\|_{W^{1,q}(\Omega)} + \|v\|_{W^{1,q}(\Omega)})
\end{align}
for all $u,v \in W^{1,q}(\Omega)$ and
the corresponding classical solution $w$ of \eqref{pr:elli}.

From Lemma~\ref{log:Sobolev}, we have
\begin{align}\label{BMO}
    \|D^2 w\|_{L^{\infty}(\Omega)}
    \le C_2 \left(1 + \|D^2 w\|_{{\rm BMO}(\Omega)}
        (1 + \log^{+} \|D^2 w\|_{W^{1,q}(\Omega)}) \right)
\end{align}
for some $C_2 > 0$.
In addition,
Lemma~\ref{bmo:elli}
entails the existence of $C_3 > 0$ such that
\begin{align}
    \|D^2 w\|_{{\rm BMO}(\Omega)}
    \le C_3 (\|u\|_{L^{\infty}(\Omega)} + \|v\|_{L^{\infty}(\Omega)}),
      \label{Singular1}
\end{align}
Here, we apply the Gagliardo--Nirenberg inequality
and the Young inequality to see that with some $C_4>0$,
\begin{align}
\notag
    &\|u\|_{W^{1,q}(\Omega)} + \|v\|_{W^{1,q}(\Omega)}
\\ \notag
    &\quad\, \le C_4 (\|∇u\|_{L^q(\Omega)} + \|∇v\|_{L^q(\Omega)}
        + \|u\|_{L^1(\Omega)} + \|v\|_{L^1(\Omega)})
\\
    &\quad\, \le C_4 (2 + \|u\|_{L^1(\Omega)} + \|v\|_{L^1(\Omega)})
      (1 + \|∇u\|_{L^q(\Omega)}^q
	       + \|∇v\|_{L^q(\Omega)}^q)
       \label{Singular3}
\end{align}
holds for every $u,v\in W^{1,q}(\Omega)$.
Combining \eqref{Singular2}, \eqref{BMO}, \eqref{Singular1}
and \eqref{Singular3} yields the result.
\end{proof}

\section{Parabolic--parabolic--elliptic case}\label{PPE}

The main purpose of this section is
to find a common existence time of solutions to \eqref{2sp-LV},
regardless of the value of $d_1$ and $d_2$ (Lemma~\ref{2sp-enough}),
which acts as the basis for
considering the approximation of solutions in \eqref{HHE-LV}.
We also prepare some lemmas about the estimate of solutions
which will be used in Section~\ref{HHE}.

%
\subsection{Some basic results of solutions to \tops{\eqref{2sp-LV}}{(1.1)}}\label{subsec:locex}

We begin this section with following local well-posedness
of solution to \eqref{2sp-LV}.

%
%
\begin{lem}\label{local-2sp}
Let $d_1, d_2, d_3, \chi_1, \chi_2,
  \mu_1, \mu_2, a_1, a_2, \alpha, \beta, \gamma \ge 0$,
and suppose that $u_0 \in C^0(\Ombar)$
and $v_0 \in C^0(\Ombar)$ are nonnegative.
Then there exist $\Tmaxd \in (0, \infty]$ and
a uniquely determined triple
$(u,v,w)$ of functions
\begin{equation}\label{ex:regularity}
  (u, v, w) \in \big(
    C^0(\Ombar \times [0, \Tmaxd)) \cap
    C^{2,1}(\Ombar \times (0, \Tmaxd))
    \big)^3
\end{equation}
which solves \eqref{2sp-LV} classically in
$\Omega \times (0, \Tmaxd)$, and is such that
\[
  \mbox{either}\ \Tmaxd = \infty
    \quad\mbox{or}\quad
  \limsup_{t \nearrow \Tmaxd}
  (\|u(\cdot,t)\|_{L^{\infty}(\Omega)} + 
    \|v(\cdot,t)\|_{L^{\infty}(\Omega)}) = \infty.
\]
Moreover, $u$, $v$ and $w$ are nonnegative
in $\Ombar \times (0, \Tmaxd)$ and belong to $C^\infty(\Om\times (0,\Tmaxd))$.
\end{lem}
\begin{proof}
This local existence result can be proved by
a standard fixed point argument as in \cite[Lemma 2.1]{STW-2014}.
The nonnegativity of $u$, $v$ and $w$ is
a consequence of the maximum principle.
\end{proof}

The regularity of solutions in Lemma~\ref{local-2sp} is the usual for classical solutions of parabolic equations. 
Some of the essential estimates in the following will be based on an ODE comparison argument employed in the study of $\frac{\rm d}{{\rm d}t} \io |∇w|^q$, 
whose application requires continuity of $t\mapsto \io |∇w(\cdot,t)|^q$ at $t=0$. This continuity is not covered by \eqref{ex:regularity}. 
In \cite{L-2015,X-2020}, this problem was solved by posing higher requirements on the initial data, including a compatibility condition on $\partial_{\nu} u_0$ on the boundary (cf.~\cite[remark, p.~6]{X-2020}), which led to $C^1$ regularity up to the initial boundary.

Here we get rid of this technical condition by applying Amann's theory \cite{A-1993} to the parabolic system of equations solved by solution components $u$ and $v$ (for given $w$).

%
%
\begin{lem}\label{W1q:regularity}
In addition to the hypotheses of Lemma~\ref{local-2sp},
let $q > n$ and $u_0, v_0 \in W^{1,q}(\Omega)$.
Then the solution $(u,v,w)$ from Lemma~\ref{local-2sp}
has the following additional regularity\/{\rm :}
\begin{align*}
  u, v \in C^0([0, \Tmaxd); W^{1,q}(\Omega)).
\end{align*}
\end{lem}
\begin{proof}
For fixed $w$, we consider the system
\begin{align}\label{2sp}
  \begin{cases}
    \util_t = d_1 \Delta \util - \chi_1 ∇ \cdot (\util ∇ w)
             + \mu_1 \util (1 - \util - a_1 \vtil),
    & x\in\Omega,\ t>0, 
  \\
    \vtil_t = d_2 \Delta \vtil - \chi_2 ∇ \cdot (\vtil ∇ w)
             + \mu_2 \vtil (1 - a_2 \util - \vtil),
    & x\in\Omega,\ t>0,
  \\
    ∇\util \cdot \nu = ∇\vtil \cdot \nu = 0,
    & x\in \pa\Omega,\ t>0,
  \\
    \util(x,0) = u_0(x),\ \vtil(x,0) = v_0(x),
    & x\in \Omega.
    \end{cases}
\end{align}
\newcommand{\matr}[1]{\begin{pmatrix}#1\end{pmatrix}}
If we let
\[
 a_{jk}:=\delta_{jk} \matr{d_1&0\\0&d_2},\quad a_j :=\matr{\chi_1 \pa_j w & 0 \\ 0 & \chi_2 \pa_j w},\quad j,k\in\{1,\ldots,n\},
\]
and define
 \begin{align*}
   \mathcal{A}\omega
   &:=
     - \sum_{j, k=1}^{n} \pa_j \kl{ a_{jk}\partial_k\omega
}
     + \sum_{j = 1}^{n} \pa_j \kl{a_j\omega}
,&\qquad 
  \mathcal{F}(\omega)
  &:=
    \begin{pmatrix}
      \mu_1 \util (1 - \util - a_1 \vtil) \\
      \mu_2 \vtil (1 - a_2 \util - \vtil)
    \end{pmatrix},
\end{align*}
then for $\omega := \matr{\util\\ \vtil}$, this system can be written as
\begin{align*}
    \begin{cases}
    \omega_t + \mathcal{A}\omega = \mathcal{F}(\omega),
    & x\in\Omega,\ t>0, 
  \\
    ∇\omega \cdot \nu = 0,
    & x\in \pa\Omega,\ t>0,
  \\
    \omega(x,0) = (u_0(x), v_0(x)),
    & x\in \Omega.
    \end{cases}
\end{align*}
Now we conclude from
\cite[Theorems 14.4 and 14.6]{A-1993} that
there exist $\til{T} \in (0, \Tmaxd]$ and a uniquely determined pair
\[
  (\util, \vtil) \in \big(
    C^0([0, \til{T}); W^{1,q}(\Omega)) \cap
    C^{2,1}(\Ombar \times (0, \til{T}))
    \big)^2
\]
that solves \eqref{2sp} classically
in $\Omega \times (0, \til{T})$.
Due to uniqueness of solutions to \eqref{2sp-LV} as asserted by Lemma~\ref{local-2sp},
this together with the regularity of $(u,v)$ on $\Ombar\times(0,\Tmaxd)$
from Lemma~\ref{local-2sp} implies that
$(\util, \vtil) = (u, v)$ and $\til{T} = \Tmaxd$.
\end{proof}

As usual, boundedness of mass is important, but easily obtained.

%
%
\begin{lem}\label{L1-2sp}
Let $T > 0$, and suppose that
$(u, v, w)$ is a classical solution to \eqref{2sp-LV}
in $\Omega \times (0,T)$
with initial data $u_0 \in C^0(\Ombar)$ and $v_0 \in C^0(\Ombar)$.
Then we have
\begin{align}
    &\|u(\cdot, t)\|_{L^1(\Omega)} 
    \le \max\{|\Omega|, \|u_0\|_{L^1(\Omega)}\}
    \label{L1u}
\\
    \intertext{and}
    &\|v(\cdot, t)\|_{L^1(\Omega)} 
    \le \max\{|\Omega|, \|v_0\|_{L^1(\Omega)}\}
    \label{L1v}
\end{align}
for all $t \in (0, T)$.
\end{lem}
\begin{proof}
Integrating the first equation of \eqref{2sp-LV}
over $\Omega$
with an application of the Cauchy--Schwarz inequality yields
\begin{align*}
    \dt \io u \le \mu_1 \io u
      - \dfrac{\mu_1}{|\Omega|} \left(\io u \right)^2.
\end{align*}
Then, an ODE-comparison leads to \eqref{L1u}.
Doing the same way for $v$, we also get \eqref{L1v}.
\end{proof}

%
\subsection{Controlling the \tops{$W^{1,q}$}{W1q}-norm of a solution by its \tops{$L^{\infty}$}{L infty}-norm}\label{W1,q-Linf}

In this section we are going to derive an inequality
which is crucial for our construction of solutions to \eqref{HHE-LV}.
Note that this section contributes the most to
remove the radial symmetric setting in Theorem~\ref{main_thm}.

The following lemma parallels \cite[Lemma 3.5]{W-2014},
\cite[Lemma 3.7]{L-2015} and \cite[Lemma 2.5]{X-2020},
however, in order to give an estimate
without dependencies on $d_1$ and $d_2$,
we need an upper bound on those.

%
%
\begin{lem}\label{control_grad}
Let $q > \max\{n, 2\}$ and $d_{*} > 0$.
Let $u_0, v_0 \in W^{1,q}(\Omega)$ be nonnegative.
Then there is $C > 0$ such that
for all $d_1, d_2 \in (0, d_{*})$ and $T > 0$,
any classical solution $(u,v,w)$ of \eqref{2sp-LV}
in $\Omega \times (0,T)$
satisfies
\begin{align*}
    &\dt
    \left(\int_{\Omega} |∇ u(\cdot, t)|^q 
    + \int_{\Omega} |∇ v(\cdot, t)|^q \right)
\\
    &\quad\, \le C(1 + 
      f(t))
      \left(\int_{\Omega} |∇ u(\cdot, t)|^q 
    + \int_{\Omega} |∇ v(\cdot, t)|^q \right)
\\
    &\qquad \, + C \cdot
      f(t)
      \log
      (1 + \|∇u(\cdot,t)\|_{L^q(\Omega)}^q 
        + \|∇v(\cdot,t)\|_{L^q(\Omega)}^q)
      \left(\int_{\Omega} |∇ u(\cdot, t)|^q 
    + \int_{\Omega} |∇ v(\cdot, t)|^q \right)
\\
    &\qquad \, + C \cdot
       g(t)
\end{align*}
for all $t \in (0,T)$, where
\begin{align}
     &f(t) := \|u(\cdot,t)\|_{L^{\infty}(\Omega)}
         + \|v(\cdot,t)\|_{L^{\infty}(\Omega)}, \quad t \in (0,T)
      \label{def:f}
\intertext{and}
     &g(t) := (\|u(\cdot,t)\|_{L^{\infty}(\Omega)}
         + \|v(\cdot,t)\|_{L^{\infty}(\Omega)})
         (\|u(\cdot,t)\|_{L^{\infty}(\Omega)}^q
         + \|v(\cdot,t)\|_{L^{\infty}(\Omega)}^q), \quad t \in (0,T).
      \label{def:g}
\end{align}
\end{lem}
\begin{proof}
In view of the identity
\[
  ∇ \cdot (u ∇ w)
  = ∇ u \cdot ∇ w + u \Delta w
  = ∇ u \cdot ∇ w 
    - u \frac{\alpha u + \beta v - \gamma w}{d_3},
\]
we can rewrite the first equation of \eqref{2sp-LV} as follows:
  \begin{align*}
    u_t &= d_1 \Delta u - \chi_1 ∇ u \cdot ∇ w
          + \chi_1 u \frac{\alpha u + \beta v - \gamma w}{d_3}
          + \mu_1 u - \mu_1 u^2 - \mu_1 a_1 uv
\\
        &= d_1 \Delta u - \chi_1 ∇ u \cdot ∇ w
         + \frac{\chi_1 \alpha}{d_3} u^2
         + \frac{\chi_1 \beta}{d_3} uv
         - \frac{\chi_1 \gamma}{d_3} uw
         + \mu_1 u - \mu_1 u^2 - \mu_1 a_1 uv.
  \end{align*}
Then $u$ solves
  \begin{align}\label{nabla-ut}
\notag
    ∇ u_t &= d_1 ∇ \Delta u
                    - \chi_1 ∇ (∇ u \cdot ∇ w)
                    + 2 \frac{\chi_1 \alpha}{d_3} u ∇ u
                    + \frac{\chi_1 \beta}{d_3} u ∇ v
                    + \frac{\chi_1 \beta}{d_3} v ∇ u
\\
       &\quad\, - \frac{\chi_1 \gamma}{d_3} u ∇ w
                    - \frac{\chi_1 \gamma}{d_3} w ∇ u
                    + \mu_1 ∇ u
                    - 2 \mu_1 u ∇ u
                    - \mu_1 a_1 u ∇ v
                    - \mu_1 a_1 v ∇ u
  \end{align}
in $\Omega\times(0,T)$.
Testing \eqref{nabla-ut} by $|∇ u|^{q-2}∇ u$ yields
  \begin{align}
\notag
    \frac{1}{q} \dt \int_{\Omega} |∇ u|^q
    &= d_1 \int_{\Omega} |∇ u|^{q-2}
         ∇ u \cdot ∇ \Delta u
     - \chi_1 \int_{\Omega} |∇ u|^{q-2}
         ∇ u \cdot ∇ (∇ u \cdot ∇ w)
\\ \notag
    &\quad\, + 2 \frac{\chi_1 \alpha}{d_3}
                     \int_{\Omega} u |∇ u|^q
     + \frac{\chi_1 \beta}{d_3}
         \int_{\Omega} |∇ u|^{q-2} u ∇ u \cdot ∇ v
     + \frac{\chi_1 \beta}{d_3} \int_{\Omega} v |∇ u|^q
\\ \notag
    &\quad\, - \frac{\chi_1 \gamma}{d_3}
       \int_{\Omega} |∇ u|^{q-2} u ∇ u \cdot ∇ w
     - \frac{\chi_1 \gamma}{d_3} \int_{\Omega} w |∇ u|^q
     + \mu_1 \int_{\Omega} |∇ u|^q
\\ \notag
    &\quad\, - 2 \mu_1 \int_{\Omega} u |∇ u|^q
     - \mu_1 a_1
      \int_{\Omega} |∇ u|^{q-2} u ∇ u \cdot ∇ v
     - \mu_1 a_1 \int_{\Omega} v |∇ u|^q
\\
    &=: I_1 + I_2 + \dots + I_{11} \label{nablau-q}
  \end{align}
in $(0, T)$.
By means of the identity
$\Delta |∇ u|^2
= 2 ∇ u \cdot ∇ \Delta u + 2 |D^2 u|^2$
and integration by parts,
we obtain
  \begin{align*}
    I_1 &= \frac{d_1}{2}
              \int_{\Omega} |∇ u|^{q-2} \Delta |∇ u|^2
         - d_1 \int_{\Omega} |∇ u|^{q-2} |D^2 u|^2
\\
         &\le \frac{d_1}{2}
              \int_{\Omega} |∇ u|^{q-2} \Delta |∇ u|^2
\\
        &= \frac{d_1}{2} \int_{\pa\Omega} |∇ u|^{q-2}
              (∇ |∇ u|^2 \cdot \nu)
         - \frac{d_1(q-2)}{4} \int_{\Omega} 
             |∇ u|^{q-4} \big|∇ |∇ u|^2 \big|^2
  \end{align*}
in $(0, T)$.
Here, we apply \cite[Lemma 2.1.\ c)]{LW-2017}
to find $C_1 > 0$ such that
\[
    \int_{\pa\Omega} |∇ u|^{q-2} (∇ |∇ u|^2 \cdot \nu)
    \le \frac{q-2}{2} \int_{\Omega} 
           |∇ u|^{q-4} \big|∇ |∇ u|^2 \big|^2
    + C_1 \left(\int_{\Omega} |∇ u|^2\right)^{\frac{q}{2}}.
\]
Therefore, the H\"{o}lder inequality implies
\begin{align}\label{I1}
    I_1 \le \frac{d_1 C_1}{2}
                \left(\int_{\Omega} |∇ u|^2\right)^{\frac{q}{2}}
        \le \frac{d_{*} C_1}{2} |\Omega|^{\frac{q-2}{2}}
                \int_{\Omega} |∇ u|^q
\end{align}
in $(0, T)$.
To estimate $I_2$, we first compute
  \begin{align*}
    |∇ u|^{q-2} ∇ u \cdot ∇ (∇ u \cdot ∇ w)
    &= |∇ u|^{q-2} \sum_{j=1}^{n} \frac{\pa}{\pa x_j} 
         \left(\sum_{i=1}^{n} \frac{\pa u}{\pa x_i}
                 \frac{\pa w}{\pa x_i}\right)
         \frac{\pa u}{\pa x_j}
\\
    &= |∇ u|^{q-2}
        \left(\sum_{i,j=1}^{n} \frac{\pa^2 u}{\pa x_i \pa x_j}
          \frac{\pa u}{\pa x_j} \frac{\pa w}{\pa x_i}
     + \sum_{i,j=1}^{n} \frac{\pa u}{\pa x_i}
         \frac{\pa^2 w}{\pa x_i \pa x_j} \frac{\pa u}{\pa x_j} \right)
\\
    &= \frac{1}{q} ∇(|∇ u|^q) \cdot ∇ w
     + |∇u|^{q-2} (D^2w∇u)\cdot∇u,
  \end{align*}
so that we can see that
  \begin{align}
\notag
    I_2
    &= - \frac{\chi_1}{q}
               \int_{\Omega} ∇(|∇ u|^q) \cdot ∇ w
      - \chi_1 \int_{\Omega} 
          |∇u|^{q-2} (D^2w∇u)\cdot∇u
\\
    &=: I_{2,A} + I_{2,B} \label{I2}
  \end{align}
in $(0, T)$.
Here an integration by parts and inserting the identity
$\Delta w = - \frac{\alpha u + \beta v - \gamma w}{d_3}$
show that
  \begin{align}
\notag
    I_{2,A} &= \frac{\chi_1}{q} \int_{\Omega} |∇ u|^q \Delta w
\\ \notag
     &\le \left(\frac{\chi_1 \alpha}{q d_3} \|u\|_{L^{\infty}(\Omega)}
         + \frac{\chi_1 \beta}{q d_3} \|v\|_{L^{\infty}(\Omega)}
         + \frac{\chi_1 \gamma}{q d_3} \|w\|_{L^{\infty}(\Omega)}
         \right) \int_{\Omega} |∇ u|^q
\\
    &\le C_2(\|u\|_{L^{\infty}(\Omega)} + \|v\|_{L^{\infty}(\Omega)})
        \int_{\Omega} |∇ u|^q \label{I2A}
  \end{align}
in $(0, T)$
with some $C_2 > 0$,
where we also used Lemma~\ref{Lp-elli}.
On the other hand, invoking Lemma~\ref{D2w},
we can obtain
\begin{align}
\notag
    I_{2,B} &\le \chi_1 \|D^2 w\|_{L^{\infty}(\Omega)}
     \int_{\Omega} |∇ u|^q
\\
    &\le C_3 \left(1 +
      (\|u\|_{L^{\infty}(\Omega)} + \|v\|_{L^{\infty}(\Omega)})
      (1 + \log
      (1 + \|∇u\|_{L^q(\Omega)}^q + \|∇v\|_{L^q(\Omega)}^q)
      ) \right)
      \int_{\Omega} |∇ u|^q \label{I2B}
\end{align}
in $(0, T)$
with some $C_3 > 0$.
Now the integral $I_3$ can be estimated as
\begin{align}\label{I3}
    I_3 \le 2 \frac{\chi_1 \alpha}{d_3} \|u\|_{L^{\infty}(\Omega)}
            \int_{\Omega} |∇ u|^q
       \quad\mbox{in}\ (0, T),
\end{align}
and similarly, we can treat $I_5$ as
\begin{align}\label{I5}
  I_5 \le \frac{\chi_1 \beta}{d_3} \|v\|_{L^{\infty}(\Omega)}
         \int_{\Omega} |∇ u|^q
       \quad\mbox{in}\  (0, T).
\end{align}
Using the Young inequality, we have
\begin{align}
\notag
    I_4 + I_{10} &\le \left(\frac{\chi_1 \beta}{d_3} + \mu_1 a_1\right)
        \|u\|_{L^{\infty}(\Omega)}
        \int_{\Omega} |∇ u|^{q-1} |∇ v|
\\
    &\le \left(\frac{\chi_1 \beta}{d_3} + \mu_1 a_1\right)
        \|u\|_{L^{\infty}(\Omega)}
        \left(\frac{q-1}{q} \int_{\Omega} |∇ u|^q
           + \frac{1}{q} \int_{\Omega} |∇ v|^q \right)
   \label{I4and10}
\end{align}
in $(0, T)$.
Applying the Young inequality
and Lemma~\ref{max_reg:elli},
we deduce
  \begin{align}
\notag
    I_6 &\le \frac{\chi_1 \gamma}{d_3}
      \|u\|_{L^{\infty}(\Omega)}
      \left(\dfrac{q-1}{q} \io |∇ u|^{q} 
        + \frac{1}{q} \io |∇ w|^{q} \right)
\\
    &\le C_4 \|u\|_{L^{\infty}(\Omega)}
      \io |∇ u|^{q} 
    + C_4 |\Omega| \|u\|_{L^{\infty}(\Omega)}
     (\|u\|_{L^{\infty}(\Omega)}^q + \|v\|_{L^{\infty}(\Omega)}^q)
      \label{I6}
  \end{align}
in $(0, T)$
with some $C_4 > 0$.
We note that from the nonnegativity of $u$, $v$ and $w$,
we can infer that $I_7 \le 0$, $I_9 \le 0$ and $I_{11}\le 0$
on $(0, T)$.
Plugging \eqref{I1}, \eqref{I2}, \eqref{I3}, \eqref{I5},
\eqref{I4and10} and \eqref{I6} into \eqref{nablau-q},
and taking into account \eqref{I2A}, \eqref{I2B},
we can establish
\begin{align}
\notag
    &\frac{1}{q} \dt \int_{\Omega} |∇ u|^q
    \quad\, \le \quad C_5 (1 + \|u\|_{L^{\infty}(\Omega)}
      + \|v\|_{L^{\infty}(\Omega)})
      \int_{\Omega} |∇ u|^q
\\ \notag
    &\qquad\, + C_5 \big( (
        \|u\|_{L^{\infty}(\Omega)} + \|v\|_{L^{\infty}(\Omega)})
        \log
      (1 + \|∇u\|_{L^q(\Omega)}^q + \|∇v\|_{L^q(\Omega)}^q)
      \big) \int_{\Omega} |∇ u|^q
\\
    &\qquad\, + C_5 \|u\|_{L^{\infty}(\Omega)}
      \int_{\Omega} |∇ v|^q
      + C_5 \|u\|_{L^{\infty}(\Omega)}
     (\|u\|_{L^{\infty}(\Omega)}^q + \|v\|_{L^{\infty}(\Omega)}^q)
     \label{Lq-nablau}
\end{align}
in $(0, T)$
with some $C_5 > 0$.
Similarly, for $v$ we can obtain
\begin{align}
\notag
    &\frac{1}{q} \dt \int_{\Omega} |∇ v|^q
\quad\, \le \quad C_{6} (1 + \|u\|_{L^{\infty}(\Omega)}
      + \|v\|_{L^{\infty}(\Omega)})
      \int_{\Omega} |∇ v|^q
\\ \notag
    &\qquad\, + C_{6} \big( (
        \|u\|_{L^{\infty}(\Omega)} + \|v\|_{L^{\infty}(\Omega)})
        \log
      (1 + \|∇u\|_{L^q(\Omega)}^q + \|∇v\|_{L^q(\Omega)}^q)
      \big) \int_{\Omega} |∇ v|^q
\\
    &\qquad\, + C_{6} \|u\|_{L^{\infty}(\Omega)}
      \int_{\Omega} |∇ u|^q
      + C_{6} \|v\|_{L^{\infty}(\Omega)}
     (\|u\|_{L^{\infty}(\Omega)}^q + \|v\|_{L^{\infty}(\Omega)}^q)
     \label{Lq-nablav}
\end{align}
in $(0, T)$
with some $C_{6} > 0$.
Combining
\eqref{Lq-nablau} and \eqref{Lq-nablav},
we observe that
there is some $C_{7} > 0$ such that
\begin{align*}
&\dt
    \left(\int_{\Omega} |∇ u(\cdot, t)|^q 
    + \int_{\Omega} |∇ v(\cdot, t)|^q \right)
\\
    &\quad\, \le C_{7}(1 + 
      f(t))
      \left(\int_{\Omega} |∇ u(\cdot, t)|^q 
    + \int_{\Omega} |∇ v(\cdot, t)|^q \right)
\\
    &\qquad \, + C_{7} \cdot
      f(t)
     \log
      (1 + \|∇u(\cdot,t)\|_{L^q(\Omega)}^q + \|∇v(\cdot,t)\|_{L^q(\Omega)}^q)
      \left(\int_{\Omega} |∇ u(\cdot, t)|^q 
    + \int_{\Omega} |∇ v(\cdot, t)|^q \right)
\\
    &\qquad \, + C_{7} \cdot
       g(t)
\end{align*}
for all $t \in (0, T)$,
where the functions $f$ and $g$ are defined as in
\eqref{def:f} and \eqref{def:g}.
Hence the proof is completed.
%
\end{proof}

Using the differential inequality in Lemma~\ref{control_grad},
we can now control the $W^{1,q}$-norm of solutions to \eqref{2sp-LV}
by their $L^{\infty}$-norm.
In the next lemma, 
with the intention of obtaining an estimate
that helps to construct solutions to \eqref{HHE-LV},
we will derive additional estimates
before using an comparison argument,
as opposed to
\cite[Corollary 3.6]{W-2014} and \cite[Lemma 3.8]{L-2015}.

We note that
thanks to Lemma~\ref{W1q:regularity}, we see that the functions
$t\mapsto \|∇u(\cdot,t)\|_{L^q(\Omega)}$ and
$t\mapsto \|∇v(\cdot,t)\|_{L^q(\Omega)}$
are continuous at $t = 0$,
and we don't have to assume that $u_0$ and $v_0$
are compatible,
which was required in \cite[Corollary 3]{L-2015} and \cite[Lemma 2.5]{X-2020}. 

%
%
\begin{lem}\label{kari}
Under the assumptions of Lemma~\ref{control_grad},
we have
\begin{align}
\notag
    &e + \io |∇u(\cdot, t)|^q + \io |∇v(\cdot, t)|^q
\\
    &\quad\, \le \exp\left[
      \log\left( e + \io |∇u_0|^q + \io |∇v_0|^q\right)
      \exp\left(\int_{0}^{t} h(s) \, ds\right)\right]
      \label{int_ineq}
\end{align}
for all $t \in (0, T)$, where
\begin{align}\label{def:h}
     h(t) := 2C(1 + 4 \max\set{
      \norm[\Lom\infty]{u(t)}, \norm[\Lom\infty]{v(t)}, 1}^{q+1}),
        \quad t \in (0,T)
\end{align}
with $C$ from Lemma~\ref{control_grad}.
\end{lem}
\begin{proof}
Let $y(t) := \io |∇u(\cdot, t)|^q + \io |∇v(\cdot, t)|^q$.
By Lemma~\ref{control_grad},
\begin{align}\label{di:grad}
    y'(t) \le C(1 +  f(t)) y(t)
    + C\cdot f(t) \log(1 + y(t)) y(t)
    + C\cdot g(t)
\end{align}
for all $t \in (0,T)$.
Here, we define
\[
    m(t) := 4 \max\set{
      \norm[\Lom\infty]{u(t)}, \norm[\Lom\infty]{v(t)}, 1}^{q+1},
    \quad t \in (0, T).
\]
Then we have $F(t) \le m(t)$ and $G(t) \le m(t)$
for all $t \in (0,T)$,
so we can estimate the right-hand side of \eqref{di:grad} as
\begin{align}
\notag
    &C(1 +  f(t)) y(t)
    + C\cdot f(t) \log(1 + y(t)) y(t)
    + C\cdot g(t)
\\ \notag
    &\quad\, \le C(1 + m(t))
      [y(t) + \log(1 + y(t))y(t) + 1]
\\
    &\quad\, \le 2C(1 + m(t)) (e + y(t))\log(e + y(t))
      \quad\mbox{for all}\ t \in (0, T),
      \label{di:grad:right}
\end{align}
where we also used the fact
$1 \le \log(e + y(t))$ for all $t \in (0, T)$.
From \eqref{def:h} we know $h(t) = 2C(1 + m(t))$, so
we can infer from \eqref{di:grad} and \eqref{di:grad:right} that
\[
    y'(t) \le h(t)(e + y(t))\log(e + y(t))
       \quad\mbox{for all}\ t \in (0, T).
\]
This implies
%
%
%
%
%
%
%
%
\[
    e + y(t) \le
      \exp\left(
      \log(e + y(0)) \exp\left(\int_{0}^{t} h(s) \, ds \right)
      \right)
\]
for all $t \in (0, T)$, which completes the proof.
\end{proof}

%
\subsection{Solutions to \tops{\eqref{2sp-LV}}{(1.1)} exist long enough}

We now show the following lemma which guarantees that
solutions to \eqref{2sp-LV} exist in $d_1, d_2$-independent time.
The idea of the proof is based on \cite[Lemma 3.9]{L-2015},
however, unlike in \cite{L-2015} we have derived the differential inequality
for $\io |∇u(\cdot, t)|^q + \io |∇v(\cdot, t)|^q$ in Lemma~\ref{control_grad}.
Therefore we will use this instead of the integral inequality \eqref{int_ineq}
and establish the result by a comparison argument.

%
%
\begin{lem}\label{2sp-enough}
(Cf. {\rm \cite[Lemma 3.9]{L-2015}})
Let $q > \max\{n, 2\}$,
and $d_{*} > 0$ be as same constant as in Lemma~\ref{control_grad}.
Then for all $D > 0$ there exist $T(D) > 0$ and $M(D) > 0$ such that
for any nonnegative $u_0, v_0 \in W^{1,q}(\Omega)$ with
\begin{align}\label{initial:W1q:bound}
    \|u_0\|_{W^{1,q}(\Omega)}, \|v_0\|_{W^{1,q}(\Omega)} \le D,
\end{align}
for all $d_1, d_2 \in (0, d_{*})$
the classical solution $(u,v,w)$ of \eqref{2sp-LV} exists
on $\Omega \times (0, T(D))$ and satisfies
\[
    \|u\|_{L^{\infty}(\Omega \times (0, T(D)))}
    + \|v\|_{L^{\infty}(\Omega \times (0, T(D)))} \le M(D).
\]
\end{lem}
\begin{proof}
Fix $C > 0$ as well as in Lemma~\ref{control_grad}.
We also fix constants $C_1, C_2 > 0$ such that
\begin{align}\label{esti:MD1}
    \|\psi\|_{L^{\infty}(\Omega)}
    \le C_1 \|∇\psi\|_{L^q(\Omega)} + C_1 \|\psi\|_{L^1(\Omega)}
  \quad\mbox{and}\quad
    \|\psi\|_{L^1(\Omega)} \le C_2 \|\psi\|_{W^{1,q}(\Omega)}
\end{align}
for all $\psi \in W^{1,q}(\Omega)$, and let
\begin{align*}
    C_3(D) := \max\{|\Omega|, C_2 D\}.
\end{align*}
We can infer from continuity and an argument based on
local well-posedness that
there is $T(D) > 0$ such that the problem
\[
    \begin{cases}
    y_{D}'(t) = 
    2C C_1 y_{D}^{2+\frac{1}{q}}(t)
        + 2C C_1 C_3(D) y_{D}^{2}(t)
      + 2CC_1(3^q C_1^q + 1) y_{D}^{1+\frac{1}{q}}(t)
\\
      \qquad \quad\quad\, + C(2C_1 C_3(D) + 1) y_{D}(t) + C(2 C_1 C_3(D))^{q+1},
  \\
    y_{D}(0) = 2D^q
  \end{cases}
\]
possesses a solution $y_{D}$, which satisfies
\begin{align}\label{comp:esti}
    y_{D}(t) \le 2D^q + 1
    \quad\mbox{for all}\ t \in (0, T(D)).
\end{align}
Now, we let $u_0, v_0 \in W^{1,q}(\Omega)$ be nonnegative
and satisfy \eqref{initial:W1q:bound},
and for fixed $d_1, d_2 \in (0, d_{*})$
we let $(u,v,w)$ be the corresponding solution of \eqref{2sp-LV}.
In light of Lemma~\ref{local-2sp},
it is sufficient to show boundedness of
$\|u\|_{L^{\infty}(\Omega)} + \|v\|_{L^{\infty}(\Omega)}$
on $(0, T(D))$.
To see this, we use the differential inequality obtained in
Lemma~\ref{control_grad}.
Here, by Lemma~\ref{L1-2sp} we have
\begin{align}\label{esti:MD2}
    \|u(\cdot, t)\|_{L^1(\Omega)} + \|v(\cdot, t)\|_{L^1(\Omega)}
    \le 2 C_3 (D)
  \quad\mbox{for all}\ t \in (0, T(D)),
\end{align}
which implies that
\begin{align}\label{fD}
\notag
    f(t) &\le C_1 (\|∇u(\cdot, t)\|_{L^q(\Omega)}
        + \|∇v(\cdot, t)\|_{L^q(\Omega)})
        + C_1 (\|u(\cdot, t)\|_{L^1(\Omega)} 
          + \|v(\cdot, t)\|_{L^1(\Omega)})
\\
    &\le C_1 (\|∇u(\cdot, t)\|_{L^q(\Omega)}
        + \|∇v(\cdot, t)\|_{L^q(\Omega)})
        + 2 C_1 C_3(D)
\end{align}
for all $t \in (0, T(D))$.
Moreover, we infer from \eqref{fD} that
\begin{align}\label{gD}
    g(t) \le [f(t)]^{q+1}
    \le 3^q C_1^{q+1}
      (\|∇u(\cdot, t)\|_{L^q(\Omega)}^{q+1}
        + \|∇v(\cdot, t)\|_{L^q(\Omega)}^{q+1})
      + 3^q (2 C_1 C_3(D))^{q+1}
\end{align}
for all $t \in (0, T(D))$.
Here we abbreviate
$y(t) := \io |∇u(\cdot, t)|^q + \io |∇v(\cdot, t)|^q$.
Then we have
$\|∇u(\cdot, t)\|_{L^q(\Omega)} \le y^{\frac{1}{q}}(t)$
and
$\|∇v(\cdot, t)\|_{L^q(\Omega)} \le y^{\frac{1}{q}}(t)$,
so as a consequence of \eqref{fD} and \eqref{gD},
Lemma~\ref{control_grad} entails that
\begin{align}
\notag
    y'(t) &\le C(1 + 2 C_1 y^{\frac{1}{q}}(t) + 2C_1 C_3(D))y(t)
\\ \notag
    &\quad + C(2 C_1 y^{\frac{1}{q}}(t) + 2C_1 C_3(D)) \log(1 + y(t)) y(t)
\\ \notag
    &\quad + C(2\cdot 3^q C_1^{q+1} y^{\frac{q+1}{q}}(t)
        + 3^q (2C_1 C_3(D))^{q+1})
\\ \notag
    &\le 2C C_1 y^{2+\frac{1}{q}}(t) + 2C C_1 C_3(D) y^2(t)
      + 2CC_1(3^q C_1^q + 1) y^{1+\frac{1}{q}}(t)
\\
    &\quad + C(2C_1 C_3(D) + 1) y(t) + C(2 C_1 C_3(D))^{q+1}
    \label{comp:diff1}
\end{align}
for all $t \in (0, T(D))$,
where we also used the fact that
$\log (1 + z) \le z$ for any $z \ge 0$.
Additionally,
\begin{align}\label{comp:diff2}
    y(0) = \io |∇u_0|^q + \io |∇v_0|^q \le 2 D^q.
\end{align}
According to \eqref{comp:diff1} and \eqref{comp:diff2},
by an ODE comparison and \eqref{comp:esti},
we therefore conclude that
\[
    y(t) \le y_{D}(t) \le 2D^q + 1
\]
for all $t \in (0, T(D))$.
Along with \eqref{esti:MD1} and \eqref{esti:MD2},
we can see that
\[
    \|u(\cdot, t)\|_{L^{\infty}(\Omega)}
    + \|v(\cdot, t)\|_{L^{\infty}(\Omega)}
    \le 2C_1 (2D^q + 1)^{\frac{1}{q}} + 2C_1 C_3 (D)
    =: M(D)
\]
for all $t \in (0, T(D))$.
\end{proof}

%
\subsection{Boundedness of \tops{$u_t$ and $v_t$}{ut and vt}}

As a preparation, we will show some regularity of
the time derivatives of
bounded solutions in an appropriate space
in the following lemma.
The statement and the proof parallel \cite[Lemma 3.10]{L-2015}.
%
%
\begin{lem}\label{utvt}
(Cf. {\rm \cite[Lemma 3.10]{L-2015}})
Let $q > \max\{n, 2\}$, $T > 0$, $M > 0$
and $u_0, v_0 \in W^{1,q}(\Omega)$.
Then for all $d_{*} > 0$ and $p \in (1,\infty)$
there is $C > 0$ such that
the following holds for $d_1, d_2 \in (0, d_{*})$\/{\rm :}
If a solution
$(u,v,w)$
of \eqref{2sp-LV}
in $\Omega \times (0,T)$
fulfills
\[
    |u(x,t)| + |v(x,t)| \le M
\]
for all $(x, t) \in \Omega \times (0, T)$, then
\[
    \|u_t\|_{L^p((0,T); \, (W^{1, \frac{q}{q-1}}(\Omega))^{*})}
    \le C
      \quad\mbox{and}\quad
    \|v_t\|_{L^p((0,T); \, (W^{1, \frac{q}{q-1}}(\Omega))^{*})}
    \le C.
\]
\end{lem}
\begin{proof}
We multiply the first equation of \eqref{2sp-LV} by $\psi\in C^1(\Ombar)$
and integrate over $\Omega$ to obtain
  \begin{align}\label{ut:psi}
\notag
    \left| \io u_t \psi \right|
    &\le d_1 \left| \io ∇u \cdot ∇\psi \right|
        + \chi_1 \left| \io u ∇w \cdot ∇\psi \right|
\\
    &\quad\, + \mu_1 \|u\|_{L^{\infty}(\Omega)}
          (1 + \|u\|_{L^{\infty}(\Omega)} 
             + a_1 \|v\|_{L^{\infty}(\Omega)})
           \left| \io \psi \right|.  
  \end{align}
for all $t\in (0,T)$.
By virtue of Lemma~\ref{kari} and Lemma~\ref{max_reg:elli} having bounds for $\norm[\Lom q]{∇u(\cdot,t)}$ and $\norm[\Lom q]{∇w(\cdot,t)}$, uniform in $t\in (0,T)$, we can derive that with some $C>0$, 
\[
 \left| \io u_t \psi \right|
  \le C \|\psi\|_{W^{1,\frac{q}{q-1}}(\Omega)}
\]
holds for every $\psi\in C^1(\Ombar)$ and throughout $(0,T)$, and hence $\|u_t(\cdot,t)\|_{(W^{1, \frac{q}{q-1}}(\Omega))^{*}}\le C$ for all $t\in(0,T)$. Accordingly,
    \[
  \left( \int_{0}^{T} 
      \|u_t\|_{(W^{1, \frac{q}{q-1}}(\Omega))^{*}}^{p}
  \right)^{\frac{1}{p}}
  \le C T^{\frac{1}{p}}.    
\]
The proof for $v$ proceeds analogously.
\end{proof}

\section{Hyperbolic--hyperbolic--elliptic case}\label{HHE}

In this section
we consider the local well-posedness of \eqref{HHE-LV}.
Based on the argument in Section~\ref{PPE},
we will prove that the solution to \eqref{HHE-LV}
is an approximating solution of \eqref{2sp-LV}.

The concept of the solution and limiting procedure
is similar to \cite{X-2020} and \cite{L-2015}.
Compared to these works, thanks to Section~\ref{W1,q-Linf},
some of the assumptions of initial data are removed.

Here and throughout the sequel,
we denote the solution of \eqref{2sp-LV}
with any $d_1, d_2 > 0$ by
$(u_{(d_1, d_2)}, v_{(d_1, d_2)}, w_{(d_1, d_2)})$.

%
\subsection{Definition and uniqueness of the strong \tops{$W^{1,q}$}{W1q}-solution}

Following \cite{W-2014,L-2015,X-2020}, we first give the solution concept for \eqref{HHE-LV}.

%
%
\begin{df}\label{df:W1q}
Let $d_3, \chi_1, \chi_2, \mu_1, \mu_2 > 0$, $a_1, a_2 \ge 0$,
$\alpha, \beta, \gamma > 0$ 
and $T \in (0, \infty]$.
A pair $(u, v,w)$ of functions is called a
\emph{strong $W^{1,q}$-solution} of \eqref{HHE-LV}
in $\Omega \times (0,T)$ if
\begin{enumerate}[{\rm (i)}]
\item
$u, v \in C^0(\Ombar \times [0,T)) \cap 
     L^{\infty}_{{\rm loc}}([0,T); W^{1,q}(\Omega))$
and
$w \in C^{2,0}(\Ombar \times [0,T))$,
\item
$u, v, w$ are nonnegative,
\item
$w$ classically solves
\[
    \begin{cases}
    0 = d_3 \Delta w + \alpha u + \beta v - \gamma w,
    & x\in\Omega,\ t \in (0,T),
  \\
   ∇ w \cdot \nu = 0,
    & x\in \pa\Omega,\ t \in (0,T),
  \end{cases}
\]
%
\item for any $\varphi \in C_0^\infty(\Ombar\times[0,T))$
\begin{align}\label{strong:u}
    - \int_{0}^{T} \int_{\Omega} u \varphi_t
    -\int_{\Omega} u_0 \varphi(\cdot, 0) 
   = \chi_1 \int_{0}^{T} \int_{\Omega} u ∇w\cdot∇\varphi
        + \int_{0}^{T} \int_{\Omega} 
          \mu_1 u (1 - u - a_1 v)\varphi
\end{align}
and
\begin{align}\label{strong:v}
    - \int_{0}^{T} \int_{\Omega} v \varphi_t
    -\int_{\Omega} v_0 \varphi(\cdot, 0) 
    = \chi_2 \int_{0}^{T} \int_{\Omega} v ∇w\cdot∇\varphi
        + \int_{0}^{T} \int_{\Omega} 
          \mu_2 v (1 - a_2 u - v)\varphi
\end{align}
hold true.
\end{enumerate}
%
\end{df}
%

We next claim that these $W^{1,q}$-solutions of \eqref{HHE-LV}
are unique,
and we also state continuous dependence of $W^{1,q}$-solutions to
\eqref{HHE-LV} with respect to the initial data.

%
%
\begin{lem}\label{unique:HHE}
Let $q > n$, $T \in (0, \infty]$,
and let $u_0, v_0 \in W^{1,q}(\Omega)$ with $u_0, v_0 \ge 0$.
Then problem \eqref{HHE-LV} possesses
at most one strong $W^{1,q}$-solution in $\Omega \times (0,T)$.
In addition, suppose that
$(u_{0j})_{j \in \N} \subset W^{1,q}(\Omega)$,
$(v_{0j})_{j \in \N} \subset W^{1,q}(\Omega)$
are sequences of nonnegative functions satisfying \eqref{ini:con}.
Assume that there are $(u, v, w)$, $(u_j, v_j, w_j)$ which form
the $W^{1,q}$-solution of \eqref{HHE-LV}
in $\Omega \times (0, T)$ for some $T > 0$ with initial data $u_0, v_0$,
and $u_{0j}, v_{0j}$, respectively.
Then the property \eqref{sol:con} holds for any $t \in (0, T)$.
\end{lem}
\begin{proof}
These follow from an investigation of the difference of
two solutions and Grönwall's inequality;
for details see
\cite[Lemma 4.2]{W-2014} and \cite[Lemma 4.2]{L-2015}.
\end{proof}

%
\subsection{Local existence and approximation}

In this section we will show that
the convergence of solutions to \eqref{2sp-LV}
toward strong $W^{1,q}$-solutions of \eqref{HHE-LV}.

The following lemma implies that
if the solutions of \eqref{2sp-LV} are uniformly bounded,
then we can construct strong $W^{1,q}$-solutions
by approximation.
This parallels \cite[Lemma 4.4]{L-2015}
and \cite[Lemma 3.2]{X-2020},
with small adjustments
due to the removal of approximation on initial data
and the upper bound of $d_1$ and $d_2$ in Lemma~\ref{kari}.

%
%
\begin{lem}\label{bdd:app}
(Cf. {\rm \cite[Lemma 4.4]{L-2015}})
Let $q > \max\{n, 2\}$ and
$u_0, v_0 \in W^{1,q}(\Omega)$ be nonnegative.
Suppose that $(d_{1,j})_j, (d_{2,j})_j \subset (0, \infty)$,
$T > 0$, $M > 0$ are such that
$d_{1,j} \searrow 0$ and $d_{2,j} \searrow 0$ as $j \to \infty$,
and such that
whenever $d_1 \in (d_{1,j})_j$ and $d_2 \in (d_{2,j})_j$,
for the solution
$(u_{(d_1, d_2)}, v_{(d_1, d_2)}, w_{(d_1, d_2)})$
of \eqref{2sp-LV} we have
\begin{align}\label{assume:bound_uv}
  u_{(d_1, d_2)} (x,t) + v_{(d_1, d_2)} (x,t) \le M
      \quad
        \mbox{for all}\ x \in \Omega
        \ \mbox{and}\ t \in (0,T).
\end{align}
Then there exists
a strong $W^{1,q}$-solution $(u,v,w)$ of \eqref{HHE-LV}
in $\Omega \times (0,T)$ such that
\begin{align*}
  &u_{(d_1, d_2)} \to u
    \quad\mbox{in}\ 
  C^0(\Ombar \times [0,T]),
\\
  &u_{(d_1, d_2)} \wsc u
      \quad\mbox{in}\ 
  L^{\infty}((0,T); W^{1,q}(\Omega)),  
\\
  &v_{(d_1, d_2)} \to v
      \quad\mbox{in}\ 
  C^0(\Ombar \times [0,T]),
\\
  &v_{(d_1, d_2)} \wsc v
      \quad\mbox{in}\ 
  L^{\infty}((0,T); W^{1,q}(\Omega))
      \quad\mbox{and}
\\
  &w_{(d_1, d_2)} \to w  
      \quad\mbox{in}\ 
  C^{2,0}(\Ombar \times [0,T])          
\end{align*}
as $(d_1,d_2)=(d_{1,j},d_{2,j})\to 0$.
\end{lem}
\begin{proof}
Throughout the proof, we let $d_j := (d_{1,j}, d_{2,j})$
for each $j \in \N$.
As a consequence of \eqref{assume:bound_uv},
we infer from Lemma~\ref{kari} that
for all sufficiently large $j \in \N$ we have
\begin{align}\label{sub:bound1}
\begin{cases}
  (u_{d_j})_j
    \quad\mbox{is bounded in}\ 
  L^{\infty}((0,T); W^{1,q}(\Omega)),
\\
   (v_{d_j})_j
    \quad\mbox{is bounded in}\ 
  L^{\infty}((0,T); W^{1,q}(\Omega)).
\end{cases}   
\end{align}
Furthermore, by Lemma~\ref{utvt} we see that
\begin{align}\label{sub:bound2}
\begin{cases}
  (u_{d_j t})_j
    \quad\mbox{is bounded in}\ 
  L^{p}((0,T); (W^{1, \frac{q}{q-1}}(\Omega))^{*}),
\\
   (v_{d_j t})_j
    \quad\mbox{is bounded in}\ 
  L^{p}((0,T); (W^{1, \frac{q}{q-1}}(\Omega))^{*}).
\end{cases}   
\end{align}
for some $p \in (1, \infty)$.
Hence, \cite[Lemma 4.4]{W-2014} implies that
\begin{align}\label{sub:compact}
\begin{cases}
  (u_{d_j})_j
    \quad\mbox{is relatively compact in}\ 
  C^0([0,T]; C^{\sigma}(\Ombar)),
\\
   (v_{d_j})_j
    \quad\mbox{is relatively compact in}\ 
  C^0([0,T]; C^{\sigma}(\Ombar))
\end{cases}
\end{align}
for some $\sigma \in (0,1)$.
Due to \eqref{sub:bound1}, \eqref{sub:bound2}
and \eqref{sub:compact},
for any subsequence of $(d_j)_j$
one can pick a further subsequence $({d_j}_i)_i$
thereof such that
\begin{align}
  &u_{{d_j}_i} \to u
    \quad\mbox{in}\ 
  C^0([0,T]; C^{\sigma}(\Ombar)), \label{sub:converge1}
\\
  &u_{{d_j}_i} \wsc u
    \quad\mbox{in}\ 
  L^{\infty}((0,T); W^{1,q}(\Omega)), \label{test:converge1}
\\
  &v_{{d_j}_i} \to v
    \quad\mbox{in}\ 
  C^0([0,T]; C^{\sigma}(\Ombar)), \label{sub:converge2}
\\
  &v_{{d_j}_i} \wsc v
    \quad\mbox{in}\ 
  L^{\infty}((0,T); W^{1,q}(\Omega)), \label{test:converge2}
\end{align}
as $i \to \infty$ for some $u$ and $v$.
In addition, from Lemma~\ref{max_reg:elli}
we can find $C>0$ such that
\begin{align*}
  \|w_{{d_j}_i} - w_{{d_j}_k}\|_{C^{2,0}(\Ombar \times [0,T])}
  &\le 
   \|w_{{d_j}_i} - w_{{d_j}_k}\|_{C^{2+\sigma,0}(\Ombar \times [0,T])}
\\
  &\le C(
   \|u_{{d_j}_i} - u_{{d_j}_k}\|_{C^{\sigma,0}(\Ombar \times [0,T])}
     +
   \|v_{{d_j}_i} - v_{{d_j}_k}\|_{C^{\sigma,0}(\Ombar \times [0,T])})
\end{align*}
for $i, k \in \N$.
Thus,
\eqref{sub:converge1} and \eqref{sub:converge2} ensures that
\begin{align}\label{test:converge3}
  w_{{d_j}_i} \to w
    \quad\mbox{in}\ 
  C^{2,0}(\Ombar \times [0,T])  
\end{align}
as $i \to \infty$ for some $w$.
We claim that
$(u,v,w)$ is a strong $W^{1,q}$-solution of \eqref{HHE-LV}.
To see this,
we test the first equation of \eqref{2sp-LV}
by an arbitrary
$\varphi \in C_0^{\infty}(\Ombar \times [0,T))$
to obtain
\begin{align*}
     & - \int_{0}^{T} \io u_{{d_j}_i} \varphi_t
    -\io u_0 \varphi(\cdot, 0)
   = - d_{1, j_{i}} \int_{0}^{T} \io ∇u_{{d_j}_i} \cdot ∇\varphi
     + \chi_1 \int_{0}^{T} \io u_{{d_j}_i} ∇w_{{d_j}_i}\cdot∇\varphi
\\
   &\qquad\, + \mu_1 \int_{0}^{T} \io u_{{d_j}_i} \varphi
    - \mu_1 \int_{0}^{T} \io u_{{d_j}_i}^2 \varphi
    - a_1 \mu_1 \int_{0}^{T} \io u_{{d_j}_i} v_{{d_j}_i} \varphi.
\end{align*}
By \eqref{sub:converge1}, \eqref{test:converge1}
and \eqref{test:converge3},
we can take $i \to \infty$ in above equation,
and it readily follows that $u$ satisfies \eqref{strong:u}.
In a similar way, \eqref{sub:converge2}, \eqref{test:converge2}
and \eqref{test:converge3} imply that
\eqref{strong:v} holds true for $v$, which proves the claim.
Now, due to the uniqueness statement in Lemma~\ref{unique:HHE},
we know that every sequence $(u_{d_j}, v_{d_j}, w_{d_j})$
converges to the solution $(u,v,w)$ of \eqref{HHE-LV}.
\end{proof}

Thanks to Lemma~\ref{2sp-enough},
uniform boundedness \eqref{assume:bound_uv} is 
always achieved on small time scales,
so we can prove the following local existence of
strong $W^{1,q}$-solutions,
corresponding to \cite[Lemma 4.5]{L-2015}
and \cite[Lemma 3.3]{X-2020}.
Similarly as in Lemma~\ref{bdd:app},
we can obtain the solutions to \eqref{HHE-LV}
without approximation on initial data.

%
%
\begin{lem}\label{exist:app}
(Cf. {\rm \cite[Lemma 4.5]{L-2015}})
Let $q > \max\{n, 2\}$.
Then for any $D > 0$, there exists $T(D) > 0$ such that
for any nonnegative $u_0, v_0 \in W^{1,q}(\Omega)$ satisfy
$\|u_0\|_{W^{1,q}(\Omega)}, \|v_0\|_{W^{1,q}(\Omega)} \le D$
there is a unique strong $W^{1,q}$-solution $(u,v,w)$
of \eqref{HHE-LV} in $\Omega \times (0, T(D))$.
Moreover,
$(u,v,w)$ can be approximated by solutions
$(u_{(d_1, d_2)}, v_{(d_1, d_2)}, w_{(d_1, d_2)})$
of \eqref{2sp-LV} in the sense that
\begin{align}
  &u_{(d_1, d_2)} \to u
    \quad\mbox{in}\ 
  C^0(\Ombar \times [0,T(D)]),
  \label{exist:app:1}
\\
  &u_{(d_1, d_2)} \wsc u
      \quad\mbox{in}\ 
  L^{\infty}((0,T(D)); W^{1,q}(\Omega)),  
  \label{exist:app:2}
\\
  &v_{(d_1, d_2)} \to v
      \quad\mbox{in}\ 
  C^0(\Ombar \times [0,T(D)]),
  \label{exist:app:3}
\\
  &v_{(d_1, d_2)} \wsc v
      \quad\mbox{in}\ 
  L^{\infty}((0,T(D)); W^{1,q}(\Omega))
      \quad\mbox{and}
      \label{exist:app:4}
\\
  &w_{(d_1, d_2)} \to w  
      \quad\mbox{in}\ 
  C^{2,0}(\Ombar \times [0,T(D)])     
  \label{exist:app:5}     
\end{align}
as $(d_1,d_2)\to 0$.
Furthermore, this solution satisfies
\begin{align}
\notag
    &e + \io |∇u(\cdot, t)|^q + \io |∇v(\cdot, t)|^q
\\
    &\quad\, \le \exp\left[
      \log\left( e + \io |∇u_0|^q + \io |∇v_0|^q\right)
      \exp\left(\int_{0}^{t} h(s) \, ds\right)\right]
    \label{W1q:sol:esti}
\end{align}
for a.e. $t \in (0, T(D))$,
where the function $h$ is defined as in \eqref{def:h}.
\end{lem}
\begin{proof}
We apply Lemma~\ref{2sp-enough} to find $T(D) > 0$ such that
solutions $u_{(d_1, d_2)}$ and $v_{(d_1, d_2)}$ to \eqref{2sp-LV}
with initial data $u_0$ and $v_0$ exist on $(0, T(D))$
and are bounded on that interval.
Then Lemma~\ref{bdd:app} ensures
the existence of a strong $W^{1,q}$-solution of \eqref{HHE-LV}
with the claimed approximation properties
\eqref{exist:app:1}--\eqref{exist:app:5},
and also Lemma~\ref{unique:HHE} tells us uniqueness of the solution.
The inequality \eqref{W1q:sol:esti} results from
Lemma~\ref{kari}, \eqref{exist:app:1}, \eqref{exist:app:2},
\eqref{exist:app:3} and \eqref{exist:app:4}
by using the weak lower semicontinuity of the norm.
\end{proof}

%
\subsection{Existence on maximal time intervals}

The integral inequality \eqref{W1q:sol:esti} leads us to obtain
the following result on existence and extensibility of
strong $W^{1,q}$-solutions of \eqref{HHE-LV}.
The idea of the proof is based on \cite[Lemma 4.6]{L-2015}
and \cite[Theorem 1.1]{X-2020},
however, the inequality \eqref{W1q:sol:esti} forms different from those,
and we will give the full proof for the sake of completeness.

%
%
\begin{lem}\label{HHE-local}
(Cf. {\rm \cite[Lemma 4.6]{L-2015}})
Suppose that for some $q > \max\{n, 2\}$,
$u_0, v_0 \in W^{1,q}(\Omega)$ are nonnegative.
Then there exist $\Tmax \in (0, \infty]$ and
a uniquely determined triple $(u,v,w)$ of functions
\[
\begin{cases}
  u, v \in C^0(\Ombar \times [0, \Tmax)) \cap
     L^{\infty}_{{\rm loc}}([0, \Tmax); W^{1,q}(\Omega))
       \quad\mbox{and}
\\
   w \in C^{2,0}(\Ombar \times [0, \Tmax))
\end{cases}
\]   
which form a strong $W^{1,q}$-solution of \eqref{HHE-LV}
in $\Omega \times (0, \Tmax)$, and which are such that
\begin{align}\label{HHE-criterion}
  \mbox{either}\ \Tmax = \infty
  \quad\mbox{or}\quad
  \limsup_{t \nearrow \Tmax}
    (\|u(\cdot, t)\|_{L^{\infty}(\Omega)}
      + \|v(\cdot, t)\|_{L^{\infty}(\Omega)}) = \infty.
\end{align}
\end{lem}
\begin{proof}
When applied to
$D := \|u_0\|_{W^{1,q}(\Omega)} + \|v_0\|_{W^{1,q}(\Omega)}$,
Lemma~\ref{exist:app} provides $T > 0$
and a strong $W^{1,q}$-solution $(u,v,w)$ of \eqref{HHE-LV}
in $\Omega \times (0,T)$ fulfilling
\begin{align}
    \notag
    &e + \io |∇u(\cdot, t)|^q + \io |∇v(\cdot, t)|^q
\\
    &\quad\, \le \exp\left[
      \log\left( e + \io |∇u_0|^q + \io |∇v_0|^q\right)
      \exp\left(\int_{0}^{t} h(s) \, ds\right)\right]
      \label{W1q:sol:esti2}
\end{align}
for almost all $t \in (0, T)$,
where the function $h$ is given by \eqref{def:h}, that is,
\[
    h(t) = 2C(1 + 4 \max\set{
      \norm[\Lom\infty]{u(t)}, \norm[\Lom\infty]{v(t)}, 1}^{q+1})
        \quad\mbox{for all}\ t \in (0,T)
\]
with $C$ from Lemma~\ref{control_grad}.
Consequently, the set
\[
\begin{split}
    \mathcal{S}:= \big\{ \widetilde{T} > 0 \mid \,
        &\exists \  \mbox{strong}\ W^{1,q}\mbox{-solution of}\ 
          \eqref{HHE-LV} \ \mbox{in}\ \Omega \times (0, \widetilde{T})
\\
      &\mbox{with initial data}\  u_0 \ \mbox{and}\  v_0\ 
        \mbox{that satisfies}\ \eqref{W1q:sol:esti2}\ 
        \mbox{for a.e.}\ t \in (0, \widetilde{T}) \big\}
\end{split}
\]
is nonempty, and thus $\Tmax := \sup \mathcal{S} \le \infty$
is well-defined.
According to Lemma~\ref{unique:HHE},
the strong $W^{1,q}$-solution on $\Omega \times (0, \Tmax)$
clearly exists and is unique,
so we only have to verify \eqref{HHE-criterion}.

Suppose on the contrary that
\[
    \Tmax < \infty
    \quad\mbox{and}\quad
    \limsup_{t \nearrow \Tmax}
    (\|u(\cdot, t)\|_{L^{\infty}(\Omega)}
      + \|v(\cdot, t)\|_{L^{\infty}(\Omega)}) < \infty.
\]
Then there exists $M > 0$ such that
\begin{align}\label{cont:bound}
    u(x,t) + v(x,t) \le M
    \quad\mbox{for all}\ (x, t) \in \Omega \times (0, \Tmax).
\end{align}
By the definition of $\mathcal{S}$,
we can find a null set $N \subset (0, \Tmax)$ such that
\eqref{W1q:sol:esti} holds for all $t \in (0, \Tmax) \setminus N$,
and invoking \eqref{cont:bound} we infer that
there is some $D_1 > 0$ such that
\[
    \|u(\cdot, t_0)\|_{W^{1,q}(\Omega)}
    + \|v(\cdot, t_0)\|_{W^{1,q}(\Omega)}
    \le D_1
\]
for all $t_0 \in (0, \Tmax)\setminus N$.
Now we apply Lemma~\ref{exist:app} again to see that
there exists $T(D_1) > 0$ such that
for each $t_0 \in (0, \Tmax) \setminus N$, the problem
\[
  \begin{cases}
    {\widehat{u}}_t 
        = - \chi_1 ∇ \cdot (\widehat{u} ∇ \widehat{w})
             + \mu_1 \widehat{u} 
               (1 - \widehat{u} - a_1 \widehat{v}),
    & x\in\Omega,\ t>0, 
  \\
    {\widehat{v}}_t 
        = - \chi_2 ∇ \cdot (\widehat{v} ∇ \widehat{w})
             + \mu_2 \widehat{v}
               (1 - a_2 \widehat{u} - \widehat{v}),
    & x\in\Omega,\ t>0,
  \\
    0 = d_3 \Delta \widehat{w} - \gamma \widehat{w}
        + \alpha \widehat{u} + \beta \widehat{v},
    & x\in\Omega,\ t>0,
  \\
    ∇ \widehat{u} \cdot \nu = ∇ \widehat{v} \cdot \nu 
      = ∇ \widehat{w} \cdot \nu = 0,
    & x\in \pa\Omega,\ t>0,
  \\
    \widehat{u}(x,0)=u(x, t_0),\ \widehat{v}(x,0)=v(x, t_0),
    & x\in\Omega,
  \end{cases}
\]
admits a strong $W^{1,q}$-solution
$(\widehat{u}, \widehat{v}, \widehat{w})$
in $\Omega \times (0, T(D_1))$, satisfying
\begin{align}
\notag
    &e + \io |∇\widehat{u}(\cdot, t)|^q
      + \io |∇\widehat{v}(\cdot, t)|^q
\\
    &\quad\, \le \exp\left[
      \log\left( e + \io |∇u(\cdot, t_0)|^q 
         + \io |∇v(\cdot, t_0)|^q\right)
      \exp\left(\int_{0}^{t} \widehat{h}(s) \, ds\right)\right]
    \label{W1q:sol:esti3}
\end{align}
for a.e. $t \in (0, T(D_1))$, where
\begin{align*}
    \widehat{h}(t) := 2C(1 + 4 \max\set{
      \norm[\Lom\infty]{\widehat{u}(t)},
      \norm[\Lom\infty]{\widehat{v}(t)}, 1}^{q+1}),
        \quad t \in (0,T(D_1)).
\end{align*}
Thus, choosing any $t_0 \in (0, \Tmax)\setminus N$
such that $t_0 > \Tmax - \frac{T(D_1)}{2}$,
\[
    (\widetilde{u}, \widetilde{v}, \widetilde{w})(\cdot, t) =
    \begin{cases}
    (u, v, w)(\cdot, t),
    & t \in (0, t_0),
  \\
    (\widehat{u}, \widehat{v}, \widehat{w})
    (\cdot, t - t_0),
    & t \in [t_0, t_0 + T(D_1))
  \end{cases}
\]
would define a strong $W^{1,q}$-solution of \eqref{HHE-LV}
in $\Omega \times (0, t_0 + T(D_1))$,
which obviously would satisfy \eqref{W1q:sol:esti2} for a.e. $t < t_0$.
On the other hand, for a.e. $t \in( t_0, t_0+T(D_1))$,
combining \eqref{W1q:sol:esti3} with \eqref{W1q:sol:esti2}
we would obtain
\begin{align}
\notag
    &e + \io |∇\widetilde{u}(\cdot, t)|^q 
      + \io |∇\widetilde{v}(\cdot, t)|^q
\\ \notag
    &\quad\, \le \exp\left[
      \log\left( e + \io |∇u(\cdot, t_0)|^q 
         + \io |∇v(\cdot, t_0)|^q\right)
      \exp\left(\int_{0}^{t - t_0} \widehat{h}(s) \, ds\right)\right]
\\ \notag
    &\quad\, = \exp\left[
      \log\left( e + \io |∇u(\cdot, t_0)|^q 
         + \io |∇v(\cdot, t_0)|^q\right)
      \exp\left(\int_{t_0}^{t} \widetilde{h}(s) \, ds\right)\right]
\\ \notag
    &\quad\, \le \exp\left[
      \log\left( e + \io |∇u_0|^q + \io |∇v_0|^q\right)
      \exp\left(\int_{0}^{t_0} h(s) \, ds\right)
      \exp\left(\int_{t_0}^{t} \widetilde{h}(s) \, ds\right)
      \right]
\\
    &\quad\, = \exp\left[
      \log\left( e + \io |∇u_0|^q + \io |∇v_0|^q\right)
      \exp\left(\int_{0}^{t} \widetilde{h}(s) \, ds\right)
      \right],
    \label{W1q:sol:esti4}
\end{align}
where
\(    \widetilde{h}(t) := 2C(1 + 4 \max\set{
      \norm[\Lom\infty]{\widetilde{u}(t)},
      \norm[\Lom\infty]{\widetilde{v}(t)}, 1}^{q+1}),
        \quad t \in (0,T(D_1)).
\)
The estimate \eqref{W1q:sol:esti4} shows that
$(\widetilde{u}, \widetilde{v}, \widetilde{w})$
satisfies \eqref{W1q:sol:esti} for a.e.\ $t \in (0, t_0 + T(D_1))$.
Since $t_0 + T(D_1) > \Tmax + \frac{T(D_1)}{2} > \Tmax$,
this would finally lead to a contradiction
to the definition of $\Tmax$,
and hence we conclude that \eqref{HHE-criterion} holds.
\end{proof}

The previous lemmas already demonstrate local well-posedness of \eqref{HHE-LV}.

\begin{prth1.2}
We only need to combine Lemmas~\ref{HHE-local} and \ref{unique:HHE}.
\qed
\end{prth1.2}

\section{Proof of \tops{Theorem~\ref{main_thm}}{the main theorem}}\label{conclude}

In this section we first show the blow-up result for \eqref{HHE-LV}.
The idea of deriving blow-up is based on \cite{W-2014},
and we will proceed as in \cite{L-2015} and \cite{X-2020}.
The purpose of this part is to improve the blow-up result
of \cite[Theorem 1.2]{X-2020},
and we take some different approach during the estimate
(Lemma~\ref{blowup:pre}).

After that, we use this blow-up result to prove Theorem~\ref{main_thm}.

%
\subsection{Blow-up of solutions to \tops{\eqref{HHE-LV}}{the HHE system}}

We first introduce
$L^1$-norm boundedness of solution to \eqref{HHE-LV},
which helps to obtain the conditions for initial data.

%
%
\begin{lem}\label{L1-HHE}
Let $T > 0$, and assume that
$(u,v,w)$ is a strong $W^{1,q}$-solution of \eqref{HHE-LV}
in $\Omega \times (0,T)$
with nonnegative $u_0, v_0 \in W^{1,q}(\Omega)$.
Then
\[
  \io u(\cdot, t) + \io v(\cdot, t)
  \le \max\left\{ \io u_0 + \io v_0, \, 2|\Omega|\right\}
\]
for all $t \in (0,T)$.
\end{lem}
\begin{proof}
The proof can be found in \cite[Lemma 4.1]{X-2020}, see also \cite[Lemma 4.7]{L-2015} or \cite[Lemma 4.1]{W-2014}.
\end{proof}

Now we borrow the following lemma from \cite{W-2014},
which is important to derive blow-up in \eqref{HHE-LV}.

%
%
\begin{lem}\label{finite_time}
Let $a > 0$, $b \ge 0$, $d > 0$ and $\kappa > 1$ be such that
\[
  a > \left(\dfrac{2b}{d}\right)^{\frac{1}{\kappa}}.
\]
If for some $T > 0$ the function $y \in C([0,T))$ is
nonnegative and satisfies
\[
  y(t) \ge a - bt + d \int_{0}^{t} y^{\kappa}(s) \, ds
\]
for all $t \in (0,T)$, then we necessarily have
\[
  T \le \dfrac{2}{(\kappa - 1) a^{\kappa - 1} d}.
\]
\end{lem}
\begin{proof}
\cite[Lemma 4.9]{W-2014}.
\end{proof}

In order to use this lemma,
we will prepare the following integral inequality
accompanied by lower bounds for some norm of $u$ and $v$.
During the proof,
since the strong $W^{1,q}$-solution $(u,v,w)$ lacks the regularity
of $u_t$ and $v_t$,
we choose a suitable test function as in
\cite[Lemma 4.9]{L-2015}
and start from \eqref{strong:u} and \eqref{strong:v}.

%
%
\begin{lem}\label{blowup:pre}
(Cf. {\rm \cite[Lemma 4.9]{L-2015}})
For any $p > 1$ satisfying
\eqref{condition:p:1}, \eqref{condition:p:2},
\eqref{condition:p:3} or \eqref{condition:p:4},
there is $C > 0$ such that
for any $\eta > 0$
one can find $B(p, \eta) > 0$ with the following property:
Whenever $q > 1$ and $(u, v, w)$ is a strong $W^{1,q}$-solution
of \eqref{HHE-LV} with nonnegative $u_0, v_0 \in W^{1,q}(\Omega)$
in $\Omega \times (0,T)$ for some $T > 0$, we have
\begin{align*}
    \io u^p(\cdot, t) + \io v^p(\cdot, t)
    & \ge \io u_0^p + \io v_0^p
    +(C - \eta)
           \int_{0}^{t} \io u^{p+1}
           +(C - \eta)
           \int_{0}^{t} \io v^{p+1}
\\
    &\quad\, - B(p, \eta) \int_{0}^{t} \left(\io u \right)^{p+1}
        - B(p, \eta) \int_{0}^{t} \left(\io v \right)^{p+1}   
\end{align*}
for all $t \in (0,T)$.
\end{lem}
\begin{proof}
For fixed $T_0 \in (0, T)$, $t_0 \in (0, T_0)$
and $\delta \in (0, T_0 - t_0)$,
we define $\chi_{\delta} \in W^{1, \infty}(\R)$ by
\[
    \chi_{\delta}(t) :=
    \begin{cases}
      1, & t < t_0, \\
      \frac{t_0 - t - \delta}{\delta},
      & t \in [t_0, t_0 + \delta], \\
      0, & t > t_0 + \delta.
    \end{cases}
\]
For each $\xi > 0$, the function $(u + \xi)^{p-1}$ belongs to
$L^{\infty}_{{\rm loc}}([0,T); W^{1,q}(\Omega))$,
so if we set $u(\cdot, t) := u_0$ for $t \le 0$,
then for any $h \in (0,1)$,
\[
    \varphi(x,t) := \chi_{\delta}(t) \dfrac{1}{h}
      \int_{t-h}^{t} (u(x,s) + \xi)^{p-1}\, ds,
    \quad
      (x,t) \in \Omega \times (0, T)
\]
defines a test function for \eqref{strong:u}. Arguing as in \cite[Lemma 4.9]{L-2015}, and passing to the limits $\delta\to 0$, $h\to 0$ and $\xi\to 0$ (in this order), we arrive at (cf. \cite[bottom of p.~1523]{L-2015})
\begin{align*}
    &\io u^p(\cdot, t_0) - \io u_0^p
      + \dfrac{p-1}{p} \io u_0^p - \dfrac{p-1}{p} \io u^p(\cdot, t_0)
\\
    &\quad\, \ge (p-1) \chi_1 \int_{0}^{t_0} \io
      u^{p-1} ∇w\cdot∇u
    + \mu_1 \int_{0}^{t_0} \io u^p (1 - u - a_1 v),
\end{align*}
Here we integrate by parts and employ the identity
$\Delta w = - \frac{\alpha u + \beta v - \gamma w}{d_3}$
to obtain
\begin{align}
\notag
    &\dfrac{1}{p} \io u^p(\cdot, t_0) - \dfrac{1}{p} \io u_0^p
\\ \notag
    &\quad\, \ge -\dfrac{p-1}{p} \chi_1 \int_{0}^{t_0} \io u^p \Delta w
      + \mu_1 \int_{0}^{t_0} \io u^p (1 - u - a_1 v)
\\ \notag
     &\quad\, = 
       \dfrac{\chi_1 \alpha (p-1)}{d_3 p} \int_{0}^{t_0} \io u^{p+1}
       + \dfrac{\chi_1 \beta (p-1)}{d_3 p} \int_{0}^{t_0} \io u^p v
       - \dfrac{\chi_1 \gamma (p-1)}{d_3 p} \int_{0}^{t_0} \io u^p w
\\ \notag
     &\qquad\,
       + \mu_1 \int_{0}^{t_0} \io u^p
       - \mu_1 \int_{0}^{t_0} \io u^{p+1}
       - a_1 \mu_1 \int_{0}^{t_0} \io u^p v
\\ \notag
    &\quad\, \ge
      \left(\dfrac{\chi_1 \alpha (p-1)}{d_3 p} - \mu_1 \right)
        \int_{0}^{t_0} \io u^{p+1}
      - \dfrac{\chi_1 \gamma (p-1)}{d_3 p} \int_{0}^{t_0} \io u^p w
\\
     &\qquad\, 
     + \left(\dfrac{\chi_1 \beta (p-1)}{d_3 p} - a_1 \mu_1
     \right) \int_{0}^{t_0} \io u^p v,
\label{energy:u:blowup}
\end{align}
where we also used nonnegativity of $u$.
We shall now treat the second term
on the right hand side of \eqref{energy:u:blowup}
in a different way than in \cite[Lemma 4.2]{X-2020}.
By the H\"{o}lder inequality and the embedding
$W^{1, \widetilde{p}}(\Omega) \hookrightarrow L^{p+1}(\Omega)$,
where $\widetilde{p} = \frac{n(p+1)}{n+p+1}$ if $n \ge 2$
and $\widetilde{p} > 1$ if $n = 1$,
we can find $C_1 > 0$ such that
\begin{align}\label{upw:1}
  \io u^p w \le \|u\|_{L^{p+1}(\Omega)}^{p}
    \|w\|_{L^{p+1}(\Omega)}
  \le C_1 \|u\|_{L^{p+1}(\Omega)}^{p}
    \|w\|_{W^{1,\widetilde{p}}(\Omega)}.
\end{align}
Moreover, Lemma~\ref{max_reg:elli} provides $C_2 > 0$ such that
\begin{align}
\notag
    \|w\|_{W^{1,\widetilde{p}}(\Omega)}
    &\le C_2 (\|u\|_{L^{\widetilde{p}}(\Omega)}
        + \|v\|_{L^{\widetilde{p}}(\Omega)})
\\
    &\le C_2 (\|u\|_{L^1(\Omega)}^{1-\theta}
        \|u\|_{L^{p+1}(\Omega)}^{\theta}
        + \|v\|_{L^1(\Omega)}^{1-\theta}
        \|v\|_{L^{p+1}(\Omega)}^{\theta})
\label{upw:2}
\end{align}
with some $\theta \in (0,1)$,
where we also used the interpolation inequality.
Inserting \eqref{upw:2} into \eqref{upw:1},
and using the Young inequality,
given $\eta > 0$ we can find
$C_3 = C_3(p, \eta) > 0$,
$C_4 = C_4(p, \eta) > 0$
and $C_5 = C_5(p, \eta) > 0$ such that
\begin{align}
\notag
    &\dfrac{\chi_1 \gamma (p-1)}{d_3 p} \io u^p w
\\ \notag
    &\quad\, \le \dfrac{\chi_1 \gamma (p-1)}{d_3 p} C_1 C_2
      (\|u\|_{L^1(\Omega)}^{1-\theta}
        \|u\|_{L^{p+1}(\Omega)}^{\theta + p}
        + \|u\|_{L^{p+1}(\Omega)}^{p}
        \|v\|_{L^1(\Omega)}^{1-\theta}
        \|v\|_{L^{p+1}(\Omega)}^{\theta})
\\ \notag
    &\quad\, \le \f{\eta}{3p} \io u^{p+1} + C_3 \left(\io u\right)^{p+1}
      + \f{\eta}{3p} \io u^{p+1}
      + C_4 \|v\|_{L^1(\Omega)}^{(1-\theta)(p+1)}
        \|v\|_{L^{p+1}(\Omega)}^{\theta(p+1)}
\\
    &\quad\, \le \f{2\eta}{3p} \io u^{p+1} + \f{\eta}{3p} \io v^{p+1}
      + C_5 \left(\io u\right)^{p+1} + C_5 \left(\io v\right)^{p+1}.
\label{upw}
\end{align}
For the third term on the right hand side of \eqref{energy:u:blowup},
we first consider the case when $p$ satisfies
\eqref{condition:p:1} or \eqref{condition:p:2} so that 
\(
    \dfrac{\chi_1 \beta (p-1)}{d_3 p} - a_1 \mu_1 > 0.
\)
Inserting this into \eqref{energy:u:blowup}
along with \eqref{upw}, we achieve
\begin{align}
\notag
        \dfrac{1}{p} \io u^p(\cdot, t_0) - \dfrac{1}{p} \io u_0^p
&\ge 
      \left(\dfrac{\chi_1 \alpha (p-1)}{d_3 p} - \mu_1
         - \f{2\eta}{3p} \right) \int_0^{t_0}\io u^{p+1}
      - \f{\eta}{3p}
          \int_{0}^{t_0} \io v^{p+1}
\\
    &\quad\, - C_5 \int_{0}^{t_0} \left(\io u\right)^{p+1}
        - C_5 \int_{0}^{t_0} \left(\io v\right)^{p+1}.
        \label{up2}
\end{align}
Next, we consider the case when $p$ fulfills
\eqref{condition:p:3} or \eqref{condition:p:4} and hence $\dfrac{\chi_1 \beta (p-1)}{d_3 p} - a_1 \mu_1 \le 0$.
We apply the Young inequality to see that
\begin{align}\label{upv}
    \io u^p v \le \dfrac{p}{p+1} \io u^{p+1}
        + \dfrac{1}{p+1} \io v^{p+1}.
\end{align}
Combining \eqref{upw} and \eqref{upv}
with \eqref{energy:u:blowup} yields
\begin{align}
\notag
    &\dfrac{1}{p} \io u^p(\cdot, t_0) - \dfrac{1}{p} \io u_0^p
\\ \notag
    &\quad\, \ge 
      \left(\dfrac{\chi_1 \alpha (p-1)}{d_3 p} - \mu_1
      + \dfrac{\chi_1 \beta (p-1)}{d_3 (p + 1)}
        - \dfrac{a_1 \mu_1 p}{p+1} - \f{2\eta}{3p} \right) \int_0^{t_0}\io u^{p+1}
\\ \notag
     &\qquad\, + \left(
     \dfrac{\chi_1 \beta (p-1)}{d_3 p (p + 1)}
     - \dfrac{a_1 \mu_1}{p+1} - \f{\eta}{3p}\right)
          \int_{0}^{t_0} \io v^{p+1}
\\
    &\qquad\, - C_5 \int_{0}^{t_0} \left(\io u\right)^{p+1}
        - C_5 \int_{0}^{t_0} \left(\io v\right)^{p+1}.
\label{up}
\end{align}
Similarly, when $p$ satisfies \eqref{condition:p:1}
or \eqref{condition:p:3}, we have
\begin{align}
\notag
        \dfrac{1}{p} \io v^p(\cdot, t_0) - \dfrac{1}{p} \io v_0^p
&\ge 
      \left(\dfrac{\chi_2 \beta (p-1)}{d_3 p} - \mu_2
         - \f{2\eta}{3p} \right) \int_0^{t_0}\io v^{p+1}
      - \f{\eta}{3p}
          \int_{0}^{t_0} \io u^{p+1}
\\
    &\quad\, - C_6 \int_{0}^{t_0} \left(\io v\right)^{p+1}
        - C_6 \int_{0}^{t_0} \left(\io u\right)^{p+1}
        \label{vp2}
\end{align}
with some $C_6 = C_6(p, \eta) > 0$, whereas when $p$ satisfies
\eqref{condition:p:2} or \eqref{condition:p:4}, it follows that
\begin{align}
\notag
    &\dfrac{1}{p} \io v^p(\cdot, t_0) - \dfrac{1}{p} \io v_0^p
\ge 
      \left(\dfrac{\chi_2 \beta (p-1)}{d_3 p} - \mu_2
      {}+ \dfrac{\chi_2 \alpha (p-1)}{d_3 (p + 1)}
        - \dfrac{a_2 \mu_2 p}{p+1} - \f{2\eta}{3p} \right) \int_0^{t_0}\io v^{p+1}
\\ 
     &\, + \left(
     \dfrac{\chi_2 \alpha (p-1)}{d_3 p (p + 1)}
     - \dfrac{a_2 \mu_2}{p+1} - \f{\eta}{3p}\right)
          \int_{0}^{t_0} \io u^{p+1}
- C_7 \int_{0}^{t_0} \left(\io v\right)^{p+1}
        - C_7 \int_{0}^{t_0} \left(\io u\right)^{p+1}.
\label{vp}
\end{align}
%
%
%
with some $C_7 = C_7(p, \eta) > 0$.
The claim thus follows from
\eqref{up2}, \eqref{up}, \eqref{vp2} and \eqref{vp}, each of which ensures positivity of the coefficients of $\int\int u^{p+1}$ and $\int\int v^{p+1}$ in the respective cases.
\end{proof}

Using Lemmas~\ref{finite_time} and \ref{blowup:pre},
we now infer from \eqref{HHE-criterion} in Lemma~\ref{HHE-local}
that corresponding solution may blow up in finite time.

%
%
\begin{lem}\label{blowup:result}
For all $p > 1$ satisfying
\eqref{condition:p:1}, \eqref{condition:p:2},
\eqref{condition:p:3} or \eqref{condition:p:4},
there exists $C(p) > 0$ with the following property\/{\rm :}
Whenever $q > \max\{n, 2\}$ and $u_0, v_0 \in W^{1,q}(\Omega)$
are nonnegative such that \eqref{condition:initial} is valid,
the strong $W^{1,q}$-solution $(u,v,w)$ blows up in finite time\/{\rm ;}
that is, in Lemma~\ref{HHE-local}, we have $\Tmax < \infty$ and
\begin{align}\label{HHE-blowup}
    \limsup_{t \nearrow \Tmax}
        (\|u(\cdot, t)\|_{L^{\infty}(\Omega)}
      + \|v(\cdot, t)\|_{L^{\infty}(\Omega)}) = \infty.
\end{align}
\end{lem}
\begin{proof}
Taking $C > 0$ from Lemma~\ref{blowup:pre}, 
we can choose $\eta > 0$ as
$\eta = \frac{1}{2}C$.
We let
\[
    C(p) := \left(
      \dfrac{2^{2+\frac{1}{p}} |\Omega|^{\frac{1}{p}} 
      B(p,\eta)}{C}\right)^{\frac{1}{p+1}}\max\{1, 2|\Omega|\},
\]
where $B(p, \eta) > 0$ is the constant
provided in Lemma~\ref{blowup:pre},
and suppose on the contrary that \eqref{condition:initial} holds
for some nonnegative
$u_0, v_0 \in W^{1,q}(\Omega)$ with $q > \max\{n,2\}$,
but the corresponding strong $W^{1,q}$-solution of \eqref{HHE-LV}
is global in time, i.e., $\Tmax = \infty$.
By Lemma~\ref{blowup:pre}, we have
\begin{align*}
    \io u^p(\cdot, t) + \io v^p(\cdot, t)
    &\ge \io u_0^p + \io v_0^p
    + \frac{C}{2} \int_{0}^{t}
      \left(\io u^{p+1} + \io v^{p+1}\right)
\\
    &\quad\, - B(p, \eta) \int_{0}^{t} \left(\io u \right)^{p+1}
        - B(p, \eta) \int_{0}^{t} \left(\io v \right)^{p+1} 
\end{align*}
for all $t > 0$.
Now we let
$y(t) := \io u^p(\cdot, t) + \io v^p(\cdot, t)$ for $t \ge 0$,
so that $y$ is continuous on $[0, \infty)$
and satisfies
\begin{align*}
    y(t) &\ge y(0) + \frac{C}{2} \int_{0}^{t}
      \left(\io u^{p+1} +\! \!\io v^{p+1}\right)
    - B(p, \eta) \int_{0}^{t} \left(\io u \right)^{p+1}
        \!- B(p, \eta) \int_{0}^{t} \left(\io v \right)^{p+1}
\\
    &\ge y(0) + 2^{-1-\frac{1}{p}} |\Omega|^{-\frac{1}{p}} C
      \int_{0}^{t} y^{\frac{p+1}{p}}(s) \, ds
    - B(p, \eta) \int_{0}^{t}
    \left[\left(\io u \right)^{p+1} + \left(\io v \right)^{p+1}\right]
\end{align*}
for all $t > 0$.
Here, we infer from Lemma~\ref{L1-HHE} that
\begin{align*}
    \int_{0}^{t}
    \left[\left(\io u \right)^{p+1} + \left(\io v \right)^{p+1}\right]
    &\le \int_{0}^{t} (\max\{1, 2|\Omega|\}\widehat{m})^{p+1}
    = (\max\{1, 2|\Omega|\})^{p+1} \widehat{m}^{p+1} t
\end{align*}
for all $t > 0$, where
$\widehat{m} := \max\{1, \io u_0 + \io v_0\}$.
Therefore,
\[
    y(t) \ge y(0) + 2^{-1-\frac{1}{p}} |\Omega|^{-\frac{1}{p}} C
      \int_{0}^{t} y^{\frac{p+1}{p}}(s) \, ds
    - (\max\{1, 2|\Omega|\})^{p+1} B(p, \eta) \widehat{m}^{p+1} t
\]
for all $t > 0$.
Now taking
$a = y(0)$, $b = (\max\{1, 2|\Omega|\})^{p+1} B(p, \eta) \widehat{m}^{p+1}$,
$d = 2^{-1-\frac{1}{p}} |\Omega|^{-\frac{1}{p}} C$
and $\kappa = \frac{p+1}{p}$,
we see from \eqref{condition:initial} that
\[
    a\left(\dfrac{2b}{d}\right)^{-\frac{1}{\kappa}}
    = y(0) \left(
      \dfrac{2(\max\{1, 2|\Omega|\})^{p+1} B(p, \eta) \widehat{m}^{p+1}}
      {2^{-1-\frac{1}{p}} |\Omega|^{-\frac{1}{p}} C}
      \right)^{-\frac{p}{p+1}}
    = y(0) (C(p)\widehat{m})^{-p} > 1,
\]
and applying Lemma~\ref{finite_time}, however, we conclude that
$\Tmax$ must be finite.
The conclusion \eqref{HHE-blowup} is then an immediate
consequence of Lemma~\ref{HHE-local}.
\end{proof}

%
\subsection{Proof of main theorem}

We finally prove our main result,
corresponding to
\cite[Theorem 1.1]{L-2015} and \cite[Theorem 1.3]{X-2020},
and expanding these to the non-symmetric case.

\begin{prth1.1}
Let $u_0, v_0 \in W^{1,q}(\Omega)$ be nonnegative
and satisfy \eqref{condition:initial},
and $T > 0$ be the maximal existence time
of the corresponding solution $(u,v,w)$ of \eqref{HHE-LV}.
According to Lemma~\ref{blowup:result},
we know that $T < \infty$ and
\begin{align}\label{main-blowup}
    \limsup_{t \nearrow T}
        (\|u(\cdot, t)\|_{L^{\infty}(\Omega)}
      + \|v(\cdot, t)\|_{L^{\infty}(\Omega)}) = \infty.
\end{align}
Assume that the conclusion of Theorem~\ref{main_thm}
were not true.
Then for some $M > 0$ there would be sequences
$(d_{1,j})_j, (d_{2,j})_j \subset (0, \infty)$ such that
$d_{1,j} \searrow 0$ and $d_{2,j} \searrow 0$ as $j \to \infty$, and
\begin{align}\label{main:cont1}
    u_{(d_{1,j}, \, d_{2,j})}(x, t)
    + v_{(d_{1,j}, \, d_{2,j})}(x,t) \le M
\end{align}
for all $(x, t) \in \Omega \times (0,T)$ and $j \in \N$.
Here, Lemma~\ref{bdd:app} entails that
\begin{align}
      &u_{(d_{1,j}, \, d_{2,j})} \to \widetilde{u}
    \quad\mbox{in}\ 
  C^0(\Ombar \times [0,T]),
    \label{main:cont2}
\\
  &v_{(d_{1,j}, \, d_{2,j})} \to \widetilde{v}
      \quad\mbox{in}\ 
  C^0(\Ombar \times [0,T])
        \quad\mbox{and}
    \label{main:cont3}
\\ \notag
  &w_{(d_{1,j}, \, d_{2,j})} \to \widetilde{w}  
      \quad\mbox{in}\ 
  C^{2,0}(\Ombar \times [0,T])     
\end{align}
as $j \to \infty$,
where $(\widetilde{u}, \widetilde{v}, \widetilde{w})$
is a strong $W^{1,q}$-solution of \eqref{HHE-LV} in
$\Omega \times (0,T)$.
Due to Lemma~\ref{unique:HHE},
we know that such solutions are unique,
and therefore we have
$(\widetilde{u}, \widetilde{v}, \widetilde{w}) = (u,v,w)$,
so that in particular
$u + v = \widetilde{u} + \widetilde{v} \le M$
in $\Omega \times (0,T)$,
according to \eqref{main:cont1}, \eqref{main:cont2}
and \eqref{main:cont3}.
However, this contradicts \eqref{main-blowup},
so the proof is complete.
\qed
\end{prth1.1}

%

\footnotesize 
  \setlength{\parskip}{0pt}
  \setlength{\itemsep}{0pt plus 0.2ex}
\bibliographystyle{abbrv}

\end{document}